\documentclass[12pt]{article}

\usepackage[english]{babel}
\usepackage{lmodern}
\usepackage[a4paper,left=13mm,right=13mm,top=15mm,bottom=15mm]{geometry}   
\usepackage[affil-it]{authblk} 

\usepackage{url}
\usepackage{cite}
\usepackage{hyperref} 

\usepackage{amsmath, amsthm, amssymb}
\usepackage{mathtools}
\usepackage{mathrsfs}  
\usepackage{tikz}
\usepackage{tikz-cd}
\usetikzlibrary{arrows,arrows.meta}
\tikzcdset{arrow style=tikz, diagrams={>=stealth}}

\newtheorem{thm}{Theorem}

\theoremstyle{definition}
\newtheorem{defn}{Definition}[section]
\theoremstyle{lemma}
\newtheorem{lem}{Lemma}[section]
\theoremstyle{proposition}
\newtheorem{prop}{Proposition}[section]

\title{On the Associative Algebra Kernels and Obstruction}

\author{Zelong Li}
\affil{\footnotesize Chair of Higher Geometry and Topology, Faculty of Mechanics and Mathematics, \\ Lomonosov Moscow State University, \\ 1Leninskie Gory, Moscow, 119991, Russia} 

\begin{document}
\maketitle

\begin{abstract}
The theory of abstract kernels in non-trivial extensions for many kinds of algebraical objects, such as groups, rings and graded rings, associative algebras, Lie algebras, restricted Lie algebras, DG-algebras and  DG-Lie algebras, has been widely studied since 1940's. Gerhard Hochschild firstly treats associative algebra as an generic type in the series of kernel problems. He proves the theorem of constructing kernel by presenting many tedious relations that may lost the readers today. In this paper, we shall illustrate the formulation and recast it for Lie algebra(-oid) kernels. We also prove the independence of 3-cocycle in the case of associative algebra. Finally, we use the universal enveloping algebra of Lie algebra to reduce the difficulty of a direct construction for the derivation algebras.
\end{abstract}

\small\tableofcontents 


\section{Introduction}
Hochshild and Mori firstly consider the case for ordinary Lie algebra and one may consult the original ideas in \cite{Hoch54a}, the detailed calculations \cite{Mori53} and some miscellaneous in \cite{Alek00}. Consider a split extension of Lie algebras of $\mathfrak{g}$ by $\mathfrak{h}$:
\[
\begin{tikzcd}
0 \ar{r} &\mathfrak{h} \ar[tail]{r}{\alpha} &\mathfrak{e} \ar[two heads]{r}{\beta} &\mathfrak{g} \ar{r} &0 
\end{tikzcd}
\]
where $\beta$ is an epimorphism such that $\mathfrak{h}=ker \beta$ and $\alpha$ is then a monomorphism. By splitness there is a linear map $\gamma: \mathfrak{g} \rightarrow \mathfrak{e}$ such that $\beta \gamma=id_{\mathfrak{g}}$.

This is equivalent to say that $\mathfrak{e}$ is a semidirect product of these two Lie algebras and has a one-to-one correspondence with a homomorphism:
\begin{align*}
\varphi':\mathfrak{g} &\rightarrow Der(\mathfrak{h}) \\
x &\mapsto ad_e(-)=[\beta^{-1}(x), -]
\end{align*}  
where $e=\beta^{-1}(x)\in \mathfrak{e}$. Now $\mathfrak{e}=\mathfrak{h} \rtimes_{\varphi'} \mathfrak{g}$.

Passing through the adjoint map, the image of $\varphi'$ consists of elements called the \textbf{inner derivation}, so $\varphi'$ can be extended to a map 
\[
\varphi: \mathfrak{g} \rightarrow Der(\mathfrak{h})/ad(\mathfrak{h})      
\]
The latter quotient is called the \textbf{outer derivation algebra}. Given any homomorphism $\varphi$ in above sense, we say $\mathfrak{g}$ and $\mathfrak{h}$ are \textbf{coupled} by $\varphi$. The pair $(h, \varphi)$ is said to be a \textbf{$\mathfrak{g}$-kernel} and \textbf{extendible} if it derived from a split extension.

Not every $\mathfrak{g}$-kernel is extendible. Given an extension we should find a proper map from $\mathfrak{g}$ to $Der(\mathfrak{h})$.  

It turns out that the transversal map $\gamma$ determines a ``covering map" $\sigma$ whence defines a coupling. One can summarize the details through the following diagram:
\[
\begin{tikzcd}
  &0 \ar{r} &ad(\mathfrak{h}) \ar[tail]{r}{j}  &Der(\mathfrak{h}) \ar[two heads]{r}{\sharp} &Out(\mathfrak{h}) \ar{r} &0 \\
	&0 \ar{r} &\mathfrak{h} \ar[tail]{r}{\alpha} \ar{u}{ad} &\mathfrak{e} \ar[two heads]{r}{\beta} \ar{u}{ad} &\mathfrak{g} \ar{r} \ar[bend left, yshift=0.6ex]{l}{\gamma} \ar[dashed]{u}{\varphi} \ar[swap, dashed]{ul}{\exists \sigma^\gamma} &0 \\
	&&&&& \mathfrak{g}\wedge \mathfrak{g} \ar[swap, dashed]{uul}{R^\varphi} \ar[bend left, dashed]{ulll}{h^\gamma}
\end{tikzcd}
\]
and define
\begin{align*}
\sigma^\gamma_x(l)&=\alpha^{-1}\big(ad_{\beta^{-1}(x)}(\alpha(l))\big)\\
                  &=\alpha^{-1}\big([\beta^{-1}(x),\alpha(l)]\big)\in \mathfrak{h},
\end{align*}
for all $x\in \mathfrak{g}$ and $l\in \mathfrak{h}$. 

There is a well-known criterion for extensibility:
\begin{center}
A $\mathfrak{g}$-kernel $(\mathfrak{h},\varphi)$ is extendible $\Leftrightarrow $ a cohomology class in $H^3(\mathfrak{g}, Zh)$ derived from $\varphi$ vanishes
\end{center}

In general, a three-dimensional cohomology can arise out of the context of extension. The following pictures display the whole steps:
\[
\begin{tikzcd}
   &&& ad(\mathfrak{h}) \ar[hook]{d}   \\
	  0 \ar{r} &Z\mathfrak{h} \ar{r}{i} &L \ar{r}{ad} &Der(\mathfrak{h}) \ar{r}{\sharp} &Out(\mathfrak{h}) \ar{r} &0 \\
	 &&& \mathfrak{g} \ar[dashed]{u}{\sigma} 
	       \ar[swap]{ur}{\varphi}    
\end{tikzcd}
\]

\[
\begin{tikzcd}
   &&& ad(\mathfrak{h}) \ar[hook]{d} \ar[dashed, bend left]{dr}{\sharp} \\
	 0 \ar{r} &Z\mathfrak{h} \ar{r}{i} &\mathfrak{h} \ar{r} \ar[two heads]{ur}{ad} &Der(\mathfrak{h}) \ar{r} &Out(\mathfrak{h}) \ar{r} &0  \\
	 &&& \mathfrak{g}\wedge \mathfrak{g} \ar{u}
	                \ar[dashed]{ul}{H}
									\ar[bend right, crossing over, swap, near end, shift right=2.8ex]{uu}{R^\sigma}
									\ar[bend right, swap, yshift=1ex]{ur}{R^\varphi}
\end{tikzcd}
\]

where $R^\sigma(x_1 \wedge x_2):=[\sigma_{x_1},\sigma_{x_2}]-\sigma_{[x_1,x_2]}$ is nonzero and $\sharp \circ R^\sigma= R^\varphi=0$

\[
\begin{tikzcd}
   0 \ar{r} &Z\mathfrak{h} \ar{r}{i} &L \ar{r}{ad} &Der(\mathfrak{h}) \ar{r}{\sharp} &Out(\mathfrak{h}) \ar{r} &0 \\
   &&& \mathfrak{g} \wedge \mathfrak{g} \wedge \mathfrak{g} \ar{ull}{\Delta^\sigma H}
	                         \ar[swap]{u}{\Delta^\sigma R^\sigma}
\end{tikzcd}
\]

where $\Delta: Alt^n(\mathfrak{g},\mathfrak{h}) \rightarrow Alt^{n+1}(\mathfrak{g},\mathfrak{h})$ happen to be a ``symbolic" differential, and  
\begin{align*}
f(\sigma, H)=f(x_1 \wedge x_2 \wedge x_3)&:=\Delta^\sigma H(x_1 \wedge x_2 \wedge x_3)
\end{align*}

Moreover, $\Delta^{\sigma} R^{\sigma}=0$. 

Another more generic pattern of this kernel problem reduces to associative algebras. In \cite{Hoch46} Hochschild introduces the laborious term ``bimultiplication algebra" replacing the position of Lie-wise derivation algebra, the inner and outer ones. Part \textbf{2} and \textbf{3} provides all basic definitions and derives the target cocycle. In part \textbf{4} we follow Mackenzie to present a Maurer-Cartan form in the meaning of associative algebra so that one can see the exclusive dependence of the cocycle. For completeness, we refer to \textbf{Appendix C} in checking the classical criterion for zero obstruction and extension. This follows on Mackenzie's work in a contemporary pattern of formulation for transitive Lie algebroid in \cite{Mackz05}. To overcome the obstacle between the associative algebra and Lie algebra, in part \textbf{5} we shall build a fundamental bridge between them. In part \textbf{6} we state the main theorems and shortly sketch their proof. We write proofs of two structure theorems, especially the simplified one, in \textbf{Appendix D} and \textbf{E}. Finally, we develop their Lie-counterpart in \textbf{8}, under-organized in \cite{Hoch54a}. Note that the two consecutive \textbf{Appendix A} and \textbf{B} are real appendices in this paper, where we vainly offer the preliminary, if not being exhaustive, knowledge of classical Hochschild cohomology; see also \cite{CE56}. 

I am very indebted to Professor A. C. Mishchenko and Professor V. M. Manuilov for their constant advice to the modification of this paper.


\section{Bimultiplication Algebra and Coupling}

\begin{defn}
A {\em bimultiplication} is pair of linear mappings $(u, v)$ of $K$ into itself, satisfying the following conditions:  
\begin{align*}
u(k_1+k_2)=u(k_1)+u(k_2), &\quad v(k_1+k_2)=v(k_1)+v(k_2) \\
u(\alpha k)=\alpha u(k), &\quad v(\alpha k)=\alpha v(k) \\
\end{align*}
and
\begin{align*}
	   k_1 u (k_2) &= v (k_1) k_2 \\ 
	   u (k_1 k_2) &= u (k_1) k_2 \\ 
	   v (k_1 k_2) &= k_1 v (k_2),
\end{align*}
for any $k_1,k_2$ in $K$.
\end{defn}

If we write $\sigma=(u, v)$, then any $\sigma$ is in
$$Hom_{\mathbb{F}}(K, K) \oplus Hom_{\mathbb{F}}(K,K)^{op}$$ 
where 
$$k \mapsto \sigma k  \quad \text{and} \quad 
k \mapsto \sigma^* k:=k\sigma $$ 
We denote the family of the pairs of endomorphism of $K$ by
$$Mul(K):=\big(End(K), End(K)^{op}\big).$$ 

\begin{defn}
The operations of addition, multiplication and scalar multiplication between the pairs of endomorphisms are defined as follows:
\begin{align*}
      (u_1,v_1)+(u_2,v_2)&=(u_1+u_2,v_1+v_2)  \\
      (u_1,v_1)(u_2,v_2)&=(u_1u_2,v_2v_1)    \\
      \alpha(u,v)&=(\alpha u,\alpha v), \alpha \in R
\end{align*}
\end{defn}

In this way, the family $Mul(K)$ forms an unital associative $\mathbb{F}$-algebra, called the {\em bimultiplication algebra} of $K$. Its identity element is $(1_u, 1_v)$.

\begin{defn}
For any $k_0$ in $K$, the pair $(u_{k_0}, v_{k_0})\in Mul(K)$ of endomorphisms of $K$ is called an {\em inner bimultiplication produced by $k_0$} if it satisfies the following conditions: 
\begin{align*}
	u_{k_0} (k) &=k_0 k \\
	v_{k_0} (k) &=k k_0, 
\end{align*} 
for all $k$ in $K$. The family of inner bimultiplications produced by an element of $K$ is called the {\em inner bimultiplication algebra}. Denote it by $Inn(K)$.
\end{defn}

Let us write $(u_{k_0}, v_{k_0})=(k_0^1, k_0^2)$, which treats maps as elements for the computation purpose.

\begin{lem}
$Inn(K)$ becomes a subalgebra of $Mul(K)$. Moreover, $Inn(K)\vartriangleleft Mul(K)$.
\end{lem}
\begin{proof}
In fact,
 \begin{align*}
 (u,v)(k_0^1,k_0^2)(K)&=(u k_0^1,k_0^2 v)(K) \\
&=(u(k_0 k),v(k) k_0)   \\
&=(u(k_0)k,k u(k_0))    \\
&=(u(k_0)^1,u(k_0)^2)(k)\in Inn(K)
 \end{align*}
\end{proof}

\begin{defn}
Two bimultiplications $(u_1, v_1)$ and $(u_2, v_2)$ are called {\em permutable} if $v_2 u_1(k)=u_1 v_2(k)$ and $v_1 u_2(k)=u_2 v_1(k)$ for any $k$ in $K$.
\end{defn}

A bimultiplication $\sigma$ is said to be \textbf{self-permutable} if $\sigma(k\sigma)=\sigma k (\sigma)$. Indeed, every inner bimultiplication of form $(k_0^1, k_0^2)$ is self-permutable. In fact, we have $k_0^1 k_0^2=k_0^2 k_0^1$. In fact,we have $k_0^1 k_0^2 (k)=k_0^1 (k k_0)=k_0 k k_0=(k_0 k) k_0=k_0^2 (k_0 k)=k_0^2 k_0^1 (k)$. The set of all self-permutatble elements needs not to be an subalgebra of $Mul(K)$, nor to be a ring. One should refer the these definitions to [Mac58].

\begin{defn}
The quotient algebra $Out(K):=Mul(K)/Inn(K)$ is called the {\em outer bimultiplications algebra} of $K$.
\end{defn}

The {\em biannihilator} of $K$ is defined to be 
$$AnniK:=\{k\in K|k K=(0)=K k\} \vartriangleleft K,$$

The map 
\begin{align*}
\epsilon: K &\rightarrow Mul(K) \\
k_0 &\mapsto (u_{k_0},v_{k_0})
\end{align*}
is an algebra homomorphism. The image subset $\epsilon(K)$ consists of elements
$$\{(u_{k_0},v_{k_0})|u_{k_0}(k)=k_0 k, v_{k_0}(k) =k k_0, \forall k_0, k \in K \} $$

\begin{prop}
$Inn(K)=im\epsilon, Anni=ker\epsilon$ and $Out(K)=coker \epsilon$ such that the following sequence is exact.
\[
\begin{tikzcd}
   &&&& im\epsilon \ar[hook]{d}   \\
	 & 0 \ar{r} &ker\epsilon \ar{r} &K \ar{r}{\epsilon} &Mul(K) \ar{r}{\iota} &coker\epsilon \ar{r} &0 \\
\end{tikzcd}
\]
\end{prop}


\section{Connections and Twisted Module}

\begin{defn}
Let $A$ and $K$ be two associative $R$-algebras. An \textbf{A-connection} on $K$ is a linear map $\mu: A \rightarrow Mul(K)$. 
\end{defn}

\begin{defn}
A connection is said to be \textbf{flat} if it becomes a homomorphism of algebra. 
\end{defn}

\begin{defn}
A connection is said to be \textbf{regular} if the image $\mu(A)$ consist of permutable elements. 
\end{defn}

\begin{defn}
For any algebra $K$ and any $A$-connection $\mu$ on $K$, the pair $(K, \mu)$ is called a \textbf{representation} of $A$or an \textbf{$A$-module} provided the flatness of $\mu$.
\end{defn}

In general, a connection may lose its flatness, hence there is no module structure on $K$ (somehow being``hindered"). 

\begin{defn}
A \textbf{coupling} of $A$ is a homomorphism of algebras $\xi:A \rightarrow Out(K)$. We also say $A$ and $K$ are \textbf{coupled} by $\xi$. In this case, the pair $(K, \xi)$ is called an \textbf{A-kernel}
\end{defn}

\begin{defn}
Let $\natural$ be the natural projection. A regular $A$-connection $\mu$ such that the following diagram commutes
\[
\begin{tikzcd}
 &Mul(K) \ar{r}{\natural} &Out(K) \\
 &A \ar[dashed]{u}{\mu}  
    \ar[swap]{ur}{\xi}
\end{tikzcd}
\]
is called an \textbf{(associative) bimultiplication law} that cover $\xi$.
\end{defn} 

Note that $\xi(A)$ consists of permutable elements if and only if $\mu(A)$ does, therefore the regularity of $\xi$ follows. 

For each $a\in A$, we shall write the element in $\xi(A)$ by $\xi_a:=([u]_a, [v]_a)$, the pair of quotient endomorphism induced by the element in $\mu(A)$. We also write $\mu_a=(u_a, v_a)$ for indicating the potential $A$-actions on $K$:
\begin{align*}
&u_a: k\mapsto a\cdot_\mu k  \\
&v_a: k\mapsto k\cdot_\mu a
\end{align*}

On the other hand, mimicking Hochscild cohomology, we may define the \textbf{twisted module} $Hom(A^{\otimes^n}, K)=\Omega^n(A, K)$ for a representation $(K, \mu)$.

\begin{defn} 
A ``symbolic" differential $\Delta^\mu : \Omega^n(A,K) \rightarrow \Omega^{n+1}(A,K)$ induced by the connection $\mu$ is given by
\begin{align*}
\Delta^\mu(f)&(a_1\otimes \cdots\otimes a_{n+1})
=u_{a_1}f(a_2\otimes \cdots \otimes a_{n+1})  \\
&+\sum_{i=1}^n (-1)^i f(a_1 \otimes \cdots \otimes a_i a_{i+1}\otimes \cdots \otimes a_{n+1})+(-1)^{n+1}v_{a_{n+1}}f(a_1\otimes \cdots \otimes a_n)
\end{align*} 
where $f\in \Omega^n(A,K)$ and $\mu_a=(u_a, v_a)$ for any $a\in A$.   
\end{defn}

Note that $\Delta^\mu \Delta^\mu=0$ fails as one drops flatness generally.


\section{The Emergence of 3-Cocycles}

A non-flat bimultiplication law $\mu$ ensues a bilinear map $R^\mu : A\otimes A \rightarrow Mul(K)$ where
\[ 
R^\mu(a_1\otimes a_2)=\mu(a_1)\mu(a_2)-\mu(a_1a_2)
\] for any $a_1, a_2\in A$.

The difference $\mu(\cdot)\mu(\cdot)-\mu(\cdot \cdot)$ is non-zero and it lies in $Inn(K)$. Indeed, as the composition $\xi$ is a homomorphism, we have $\xi(a_1)\xi(a_2)-\xi(a_1a_2)=\natural\circ \mu(a_1) \iota \circ \mu(a_2)-\natural \circ \mu(a_1 a_2)=\natural \circ(\mu(a_1)\mu(a_2)-\mu(a_1 a_2))=\natural \circ R^\mu(a_1\otimes a_2)=0$. Therefore, $R^\mu$ induces a bilinear map $A\otimes A \rightarrow Inn(K)$ and we still denote it by $R^\mu$. It is called the \textbf{curvature} with respect to $\mu$.

\begin{lem} 
For every bimultiplication law $\mu$, there are bilinear mappings $h: A\otimes A \rightarrow K$ such that 
\[
\epsilon \circ h=R^\mu
\]
In other word, each $h$ naturally lifts $R^\mu$. 
\end{lem}

\begin{proof}
Recall that $\epsilon$ is surjective and the element $R^\mu(\cdot, \cdot)$ is an inner bimultiplication produced by some $k=h(\cdot,\cdot)$ in the preimage.
\[ 
\begin{tikzcd}
&Inn(K) \subseteq Mul(K) \\
K \ar{ur}{\epsilon} &A\otimes A \ar[dashed]{l}{h} \ar[swap]{u}{R^\mu} 
\end{tikzcd}
\]
\end{proof}

\begin{defn}
The map $h: A\otimes A \rightarrow K$ is called a \textbf{hindrance} of the coupling.
\end{defn}
In other words, it hinders a (non-flat) connection that covers the coupling.

Let us concretely compute the curvature by assigning an element $k$.
\begin{align*}
\mu(a_2)\mu(a_3)(k)
        &=(u,v)(u^\prime,v^\prime)(k)=(uu^\prime,v^\prime v)(k)=\bigl(u(u^\prime(k)), v^\prime(v(k))\bigr)      \\
				&=\bigl(a_2\cdot(a_3\cdot k), (k\cdot a_2)\cdot a_3 \bigr)    \\ 
\mu(a_2a_3)(k) 
        &=(u^{\prime\prime},v^{\prime\prime})(k)=(u^{\prime\prime}(k), v^{\prime\prime}(k))=\bigl((a_2a_3)\cdot k,\ k\cdot(a_2 a_3)\bigr)
\end{align*}

So, the first coordinate in the difference of above two identities is 
\[ a_2\cdot(a_3\cdot k)-(a_2a_3)\cdot k \]
and the second one is
\[(k\cdot a_2)\cdot a_3-\ k\cdot(a_2a_3) \]

Since such an inner multiplication is produced by a hindrance, we have 
\[ 
\mu(\cdot)\mu(\cdot)-\mu(\cdot\cdot)
=(uu^\prime-u^{\prime\prime},v^\prime v-v^{\prime\prime})
=\epsilon\circ h(\cdot, \cdot)
=\bigl(h(\cdot,\cdot)^1, h(\cdot,\cdot)^2\bigr)
\]

More precisely, by applying a $k$, we have
\begin{align*}
h(a_2\otimes a_3)^1(k)&=h(a_2\otimes a_3)k   \\
h(a_2\otimes a_3)^2(k)&=kh(a_2\otimes a_3)
\end{align*}

Comparing these coordinates, we get two important identities: 
\begin{align*}
a_2\cdot(a_3\cdot k)-(a_2a_3)\cdot k &=h(a_2\otimes a_3)k   
    \tag{3.1}      \\
(k\cdot a_2)\cdot a_3-k\cdot(a_2a_3) &=kh(a_2\otimes a_3)
    \tag{3.2}                  
\end{align*}

We now deduce some characteristic identities involving $h$ in detail when it takes triple variables in $A$: 
For any $k\in K, a_r, a_s, a_t\in A$, we firstly note that
\begin{align*}
a_r\cdot((a_s a_t)\cdot k)-(a_r a_s a_t)\cdot k &=h(a_r\otimes a_s a_t)k 
   \tag{3.1$a$}   \\
(a_r a_s)\cdot(a_t\cdot k)-(a_r a_s a_t)\cdot k &=h(a_ra_s\otimes a_t)k 
   \tag{3.1$b$}       
\end{align*}
by viewing $a_r a_s$ as an integral symbol and then substituting it into (3.1) in two different ways. 

And we have
\[
a_r\cdot(a_s\cdot(a_t\cdot k))-a_ra_s\cdot(a_t\cdot k)=h(a_r\otimes a_s)(a_t\cdot k)  
   \tag{3.1$c$}
\]
by viewing $a_r\cdot k$ as an integral symbol and then substituting it into (3.1) again.

Secondly, with (3.1) and then (3.1$a$), we have
\begin{align*}
a_1\cdot\bigl(a_2\cdot(a_3\cdot k)\bigr) &=a_1\cdot\bigl((a_2a_3)\cdot k+h(a_2\otimes a_3)k \bigr)   \\
&=a_1\cdot\bigl((a_2a_3)\cdot k \bigr)+(a_1\cdot h(a_2\otimes a_3))\cdot k   \\
&=(a_1a_2a_3)\cdot k+h(a_1\otimes a_2a_3)k+(a_1\cdot h(a_2\otimes a_3))k
\end{align*}

On the other hand, by (2.1$c$) and then by (2.1$b$), we have
\begin{align*}
a_1\cdot\bigl(a_2\cdot(a_3\cdot k)\bigr) &=(a_1a_2)\cdot(a_3\cdot k)+h\{a_1\otimes a_2\}(a_3\cdot k)   \\
&=(a_1a_2a_3)\cdot k+h(a_1a_2\otimes a_3)k+(h(a_1\otimes a_2)\cdot a_3)k
\end{align*}

We define a trilinear cochain with respect to $\mu_a=(u_a, v_a)$ and $h$
\[
f(a_1\otimes a_2\otimes a_3):=u_{a_1} h(a_2\otimes a_3)-h(a_1a_2\otimes a_3)+h(a_1\otimes a_2a_3)-v_{a_3}h(a_1\otimes a_2) 
\]

Multiply by $k$ on the right on each side of this formula, 
\[
f(a_1\otimes a_2\otimes a_3)\cdot k=a_1\cdot\bigl(a_2\cdot(a_3\cdot k)\bigr)-a_1\cdot\bigl(a_2\cdot(a_3\cdot k)\bigr)=0
\]

Likewise, starting from (3.2), we can compute the coboundary formula by multiplying $k$ on the left and get $k\cdot f(a_1\otimes a_2\otimes a_3)=0$. These two identities imply $f$ takes its value in the biannihilator of $K$.

\begin{lem}
The coupling defines a structure of $A$-$A$-bimodule on the biannihilator of $K$, being independent from the choice of $\mu$.
\end{lem}
\begin{proof}
Let $AnniK$ be the biannihilator of $K$.
Let $n\in AnniK$, we compute $a_1(a_2\cdot n)-(a_1a_2)\cdot n=h(a_1\otimes a_2)\cdot n=0$. Similarly, we have $(n\cdot a_2)\cdot a_3-\ n\cdot(a_2a_3)=0$, either. This proves that $AnniK$ is both left and right $A$-module.

Furthermore, we need $(a_1\cdot n)\cdot a_2=a_1\cdot(n\cdot a_2)$ so that $AnniK$ becomes a bimodule. But this is just from the formula $v_2u_1(n)-u_1v_2(n)=0$, by the permutability of $\mu$. 
\end{proof}

The coupling $\xi$ induces a representation on $AnniK$ not depending one the choice of $\mu$. Indeed, let $\mu:A\rightarrow Mul(K)$ be a covering of $\xi$, then $\mu_{AnniK}:A\rightarrow Mul(K)$ defines the mapping $\mu_{AnniK}(a): n\mapsto n$ instead of $k\mapsto k$. Now if $\mu|_{AnniK}, \mu'|_{AnniK}$ are two restricted coverings of $\varphi$, then $\mu'|_{AnniK}-\mu|_{AnniK}(a)$ is an element in $Inn(K)$, as well as in $Inn(AnniK)$. Such a inner bimultiplication is produced by some elements $n\in {AnniK}$ and therefore vanishes on $AnniK$. This concludes that $AnniK$ does not rely on the choice of $\mu$.

No cochain complex follows for the twisted module $\Omega(A, K)$. However, the restricted map $\mu|_{AnniK}: A \rightarrow Mul(AnniK)$ is a homomorphism, thus it is a flat connection. We change the notation of it by $\rho^\xi$, indicating its exclusive dependence on the choice of $\xi$.  

\begin{defn}
the {\em  central (or annihilatoral) $A$-connection} on $AnniK$. Due to the flatness the pair $(AnniK^\xi, \rho^\xi)$ is the {\em central (or annihilatoral) representation} of $A$, and the twisted module $Hom(A^{\otimes^n}, AnniK)$ becomes the standard Hochschild cochain complex $C^n(A, AnniK)=$, together with the differential $\delta^{\rho^\xi}$ (or simply $\delta^\xi$) in the usual sense.
\end{defn}

In other words, $\Delta^\mu$ collapse to $\delta^\xi$ on $AnniK$. 

Denote
$$f=\Delta^\mu(h)$$ 
a sort of taking the differential of $h$ (actually no applying $\delta^{xi}$ on $h$ as it takes values in the non-module $K$). Or we write
$$f=f(\mu, h)$$
to indicate its stem from the law-covering $\mu$ and the hindrance $h$.

Therefore, we obtain the following theorem:

\begin{thm}
$f(\mu, h) \in C^3(A,AnniK)$
\end{thm}

\begin{lem}  
$f(\mu, h) \in Z^3(A,AnniK)$
\end{lem}
\begin{proof}
In fact, 
\begin{align*}
(\delta^\xi \Delta^\mu h)(a_1\otimes a_2\otimes &a_3\otimes a_4)=a_1\cdot\Delta h(a_2\otimes a_3\otimes a_4)-\Delta h(a_1a_2\otimes a_3\otimes a_4)     \\
&+\Delta h(a_1\otimes a_2a_3\otimes a_4)-\Delta h(a_1\otimes a_2\otimes a_3a_4)+\Delta h(a_1\otimes a_2\otimes a_3)\cdot a_4
\end{align*}

Expand all $\Delta h$, 	
\begin{align*}
&=a_1\cdot \bigl(a_2\cdot h(a_3\otimes a_4)-h(a_2a_3\otimes a_4)+h(a_2\otimes a_3a_4)-h(a_2\otimes a_3)\cdot a_4 \bigr)        \\ 
&\qquad -\bigl(a_1a_2\cdot h(a_3\otimes a_4)-h(a_1a_2a_3\otimes a_4)+h(a_1a_2\otimes a_3a_4)-h(a_1a_2\otimes a_3)\cdot a_4\bigr)   \\
&\qquad +\bigl(a_1\cdot h(a_2a_3\otimes a_4)-h(a_1a_2a_3\otimes a_4)+h(a_1\otimes a_2a_3a_4)-h(a_1\otimes a_2a_3)\cdot a_4\bigr) \\ 
&\qquad -\bigl(a_1\cdot h(a_2\otimes a_3a_4)-h(a_1a_2\otimes a_3a_4)+h(a_1\otimes a_2a_3a_4)-h(a_1\otimes a_2)\cdot a_3a_4 \bigr)  \\
&\qquad +\bigl(a_1\cdot h(a_2\otimes a_3)-h(a_1a_2\otimes a_3)+h(a_1\otimes a_2a_3)-h(a_1\otimes a_2)\cdot a_3\bigr)\cdot a_4
\end{align*}

Most of the terms are canceled out, so the remaining terms are the sum of the following two terms:
\begin{align*}
& a_1\cdot(a_2\cdot h(a_3\otimes a_4))-a_1a_2\cdot h(a_3\otimes a_4)  \\
& h(a_1\otimes a_2)\cdot a_3a_4-(h(a_1\otimes a_2)\cdot a_3)\cdot a_4
\tag{$\star$}
\end{align*}

Since $h(\cdot, \cdot)\in K$, we can rewrite (2.1) and (2.2) as follows:
\begin{align*}
a_r\cdot(a_s\cdot h(\cdot,\cdot)) &-(a_ra_s)\cdot h(\cdot,\cdot)=h(a_r,a_s)h(\cdot,\cdot) 
  \tag{2.3}  \\
(h(\cdot,\cdot)\cdot a_r)\cdot a_s &-h(\cdot,\cdot)\cdot(a_ra_s)=h(\cdot,\cdot)h(a_r,a_s)
  \tag{2.4} 
\end{align*}

Therefore, by applying these rules to $(\star)$, we have shown $(\Delta_N \Delta h)=h(a_1\otimes a_2)h(a_3\otimes a_4)-h(a_1\otimes a_2)h(a_3\otimes a_4)=0$, as desired.
\end{proof}

\begin{lem}
Given a bimultiplication law $\mu$ that covers $\xi$, let $h$, $h'$ be two hindrances of $R^\mu$. Write $f=f(\mu, h)=\Delta^\mu h$ and $f'=f(\mu, h')=\Delta^\mu h'$ to be the corresponding cocycles. Then $h-h'=i\circ g$ for some $g\in C^2(A,ZK)$ and $f-f'=\delta^\xi g$.  
\end{lem}

\begin{proof}
1). Firstly we have $\epsilon \circ (h- h')=R^\mu-R^\mu=0$. Since $\epsilon(k)=0$ implies $k=i\{n\}$ and $i\{n\}k=0=ki\{n\}$, then there is a unique $g: A\otimes A \rightarrow N$ such that $h-h'=i\circ g$. \\
2).For any $a_1, a_2, a_3\in A$
We compute 
\begin{align*}
&i(f(a_1\otimes a_2 \otimes a_3)-f'(a_1\otimes a_2 \otimes a_3))  \\
&=i \circ (a_1\cdot (h-h')(a_2 \otimes a_3)-(h-h')(a_1a_2 \otimes a_3)+(h-h')(a_1\otimes a_2a_3)-(h-h')(a_1\otimes a_2)\cdot a_3)  \\
&=i \circ (a_1\cdot (i\circ g)(a_2 \otimes a_3)-(i\circ g)(a_1a_2 \otimes a_3)+(i\circ g)(a_1\otimes a_2a_3)-(i\circ g)(a_1\otimes a_2)\cdot a_3)   \\
&=i\circ \delta g(a_1\otimes a_2 \otimes a_3),  
\end{align*}
as desired.
\end{proof}


\section{The Independence of the $3$-Cocycles}
We have already seen that the construction of our target $3$-cocycle employs the coverings $\mu$ and hinderances $h$ and thus we write $f(\mu, h)$. Nevertheless, the choice of this cocycle does \textit{not} rely on the choice of $\mu$ and $h$--it only depends on the given coupling $\xi$, which we shall show in this section. Once we succeed in doing that, this cohomology $\{f\}$ will be called the \textit{obstruction} determined by $\xi$. 

Firstly, we show $f$ does not rely on the choice of $\mu$.

\begin{lem}
Let $\mu$ and $\mu'$ be two bimultiplication laws that cover $\xi$. Then 
\[
\mu'=\mu+ \epsilon \circ l 
\]
for some maps $l:A\rightarrow K$, and
\[ 
R^{\mu'}-R^\mu=\epsilon \circ(\Delta^\mu(l)+ l\cdot l)
\]
\end{lem}

\begin{proof}

Since $\epsilon$ is a homomorphism and $l(a_i)\in K$, then
\begin{align*}
&(R^{\mu'}-R^\mu)(a_1\otimes a_2)  \\
&=(u_{a_1},v_{a_1})(l(a_2)^1,l(a_2)^2)+(l(a_1 )^1,l(a_1)^2)(u_{a_2},v_{a_2})+(l(a_1)^1,l(a_1)^2)(l(a_2)^1,l(a_2)^2)   \\
&-(l(a_1 a_2)^1,l(a_1 a_2)^2)        \\
&=(u_{a_1} l(a_2)^1,l(a_2)^2 v_{a_1}))+(l(a_1)^1 u_{a_2},v_{a_2} l(a_1)^2)+(l(a_1)^1 l(a_2)^1,l(a_2)^2 l(a_1)^2)-(l(a_1 a_2)^1,l(a_1 a_2)^2) 
\end{align*}

The first coordinate is
\begin{align*}
&u_{a_1}l(a_2)^1-l(a_1 a_2)^1+l(a_1)^1 u_{a_2}+l(a_1)^1 l(a_2)^1  \\
&=a_1\cdot l(a_2)^1-l(a_1 a_2)^1+l(a_1)^1\cdot a_2+l(a_1)^1l(a_2)^1  \\
&=(\Delta^\mu (l)+l\cdot l)^1
\end{align*}

Likewise, we have the second coordinate
\begin{align*}
&u_{a_1}l(a_2)^2-l(a_1 a_2)^2+v_{a_2}l(a_1)^2+ l(a_2)^2 l(a_1)^2   \\
&=a_1\cdot l(a_2)^2-l(a_1 a_2)^2+l(a_1)^2\cdot a_2+ l(a_2)^2 l(a_1)^2 \\
&=(\Delta^\mu (l)+l\cdot l)^2
\end{align*}
Therefore, we have 
\begin{align*}
&(R^{\mu'}-R^\mu)(a_1\otimes a_2)  \\
&=((\Delta^\mu (l(a_1\cdot a_2))+l(a_1)l(a_2))^1,(\Delta^\mu (l(a_1\cdot a_2))+l(a_1)l(a_2))^2)  \\
&=\epsilon \circ [\Delta^\mu (l)+l\cdot l]
\end{align*}
\end{proof}

\begin{lem}
Let $\mu$ be a bimultiplication law that covers $\xi$ and $h$ a hindrance of $R^\mu$. Let $\mu'$ be another bimultiplication law that covers $xi$ such that $\mu'= \mu+ \epsilon \circ l$ for some maps $l: A \rightarrow K$. Then 
\[
h' = h+ (\Delta^\mu (l)+ l\cdot l)  \qquad (\star\star)
\]
is a hindrance of $R^{\mu'}$. Moreover, $f(\mu, h)=f(\mu',h')$.
\end{lem}

\begin{proof}
\begin{align*}
R^{\mu'}=\epsilon \circ h' &, R^\mu=\epsilon \circ h    \\
\epsilon \circ h'=\epsilon \circ h &+\epsilon \circ (\Delta^\mu (l)+l\cdot l), 
\end{align*}
as desired.

It means that 
\[
(h'{^1},h'{^2})(k)=(h^1,h^2)(k)+((\Delta^\mu(l)+l\cdot l)^1,(\Delta^\mu (l)+l\cdot l)^2)(k)
\]
By applying $k$ for the first coordinate, for example, we have the following expression
\[
h'(a_1\otimes a_2)\cdot k=h(a_1\otimes a_2)\cdot k+[\Delta^\mu l(a_1\cdot a_2)+l(a_1)l(a_2)]\cdot k
\]

By omitting the superscript temporarily and taking differential with respect to $\Delta^{\mu'}$, we have
\[
\Delta^{\mu'} h'^{1}=\Delta^{\mu'} h^1+ \Delta^{\mu'}(\Delta^\mu (l)+l\cdot l)^1
\]
Then
\[
\Delta^{\mu'} h'^{1}-\Delta^\mu h^1=\Delta^{\mu'} h^1-\Delta^\mu h^1+ \Delta^{\mu'}(\Delta^\mu (l)+l\cdot l)^1
\]
The LHS is 
\begin{align*}
&(\Delta^{\mu'} {h'}^{1}-\Delta^\mu h^1)(a_1\otimes a_2 \otimes a_3)  \\
&=u_{a_1}' h(a_2\otimes a_3)^1-h(a_1 a_2\otimes a_3)^1+h(a_1\otimes a_2 a_3 )^1+v_{a_3}' h(a_1\otimes a_2)^1  \\
&-u_{a_1}h(a_2\otimes a_3)^1+h(a_1 a_2\otimes a_3)^1-h(a_1\otimes a_2 a_3)^1-v_{a_3}h(a_1\otimes a_2)^1   \\
&=(u_{a_1}'-u_{a_1})h(a_2\otimes a_3)^1-(v_{a_3}'-v_{a_3})h(a_1\otimes a_2 )^1    \\
&=l(a_1)^1 h(a_2\otimes a_3)^1-h(a_1\otimes a_2)^1 l(a_3)^1
\end{align*}

And the RHS is(omit $1$)
\begin{align*}
& \Delta^{\mu'}(\Delta^\mu (l)+l\cdot l)  \\
&=\Delta^{\mu+\epsilon \circ l} (\Delta^\mu (l)+l\cdot l)  \\
&=\Delta^\mu \Delta^\mu (l)+\Delta^\mu (l\cdot l)+\Delta^{\epsilon \circ l} \Delta^\mu (l)+\Delta^{\epsilon \circ l}(l\cdot l)
\end{align*}

Let us compute each of the above four terms:
\begin{align*}
&\Delta^\mu \Delta^\mu (l)(a_1\otimes a_2\otimes a_3)   \\
&=a_1\cdot \Delta l(a_2\otimes a_3)-\Delta l(a_1 a_2\otimes a_3)+\Delta l(a_1\otimes a_2 a_3)-\Delta l(a_1\otimes a_2)\cdot a_3 \\
&=a_1\cdot (a_2\cdot l(a_3))-a_1\cdot l(a_2 a_3)+a_1\cdot (l(a_2)\cdot a_3)-a_1 a_2\cdot l(a_3)+l(a_1 a_2 a_3)-l(a_1 a_2)\cdot a_3   \\
&+a_1\cdot l(a_2 a_3)+l(a_1 )\cdot a_2 a_3-(a_1\cdot l(a_2))\cdot a_3-l(a_1 a_2 a_3)+l(a_1 a_2)\cdot a_3-(l(a_1 )\cdot a_2)\cdot a_3   \\
&=a_1\cdot (a_2\cdot l(a_3))-a_1 a_2\cdot l(a_3)+l(a_1)\cdot a_2 a_3-(l(a_1)\cdot a_2)\cdot a_3
\end{align*}
Now we view the element $l(a_i)$ as mappings(after putting $\epsilon$ in front of it). This gives us a negative part of previous:
\[  
h(a_1\otimes a_2)^1 l(a_3)^1-l(a_1)^1 h(a_2\otimes a_3)^1
\]

For $\Delta^\mu (l\cdot l)(a_1\otimes a_2\otimes a_3)$, set $l(x)l(y)=f(x,y)$, then
\begin{align*}
\Delta f(x,y,z)&=xf(y,z)-f(xy,z)+f(x,yz)-f(x,y)z  \\
&=x(l(y)l(z))-l(xy)l(z)+l(x)l(yz)-(l(x)l(y))z
\end{align*}

For $(\epsilon \circ l) \Delta^\mu (l)(a_1\otimes a_2)=l(a_1)\Delta l(a_2 \otimes a_3)-\Delta l(a_1\otimes a_2)l(a_3)$, we have   
\begin{align*}
l(x)\Delta l(y,z)&=l(x)(yl(z))-l(x)l(yz)+l(x)l(y)z   \\
-\Delta l(x,y)l(z)&=-xl(y)l(z)+l(xy)l(z)-l(x)yl(z)
\end{align*}

So $\Delta^\mu (l\cdot l)+(\epsilon \circ l) \Delta^\mu (l)=0$

Lastly, 
\begin{align*}
(\epsilon \circ l)(l\cdot l)&=l(a_1)^1 (l(a_2)l(a_3 ))-l(a_3)^2 (l(a_1 )l(a_2))   \\
&=l(a_1)l(a_2)l(a_3)-l(a_1)l(a_2)l(a_3)   \\
&=0  
\end{align*}
The second coordinate can be computed similarly, whence $f(\mu',h')-f(\mu,h)=\Delta^{\mu'}h'-\Delta^\mu h=0$, as desired.
\end{proof}

\begin{thm}
The coupling $\xi$ of $A$ defines a cohomological class in $HH^3(A,ZK, \rho^\xi)$, elements of which does not depend on the choice of the bimultiplication law $\mu$ that covers $\xi$ and the hindrance of the law.
\end{thm}

\begin{defn}
Such a class is called the \textbf{obstruction} derived from the coupling.
\end{defn} 

As $f(\mu, h)$ becomes the representative cocycle, independent from the two ``variables", we denote it by 
$$f:=f^\xi$$
and the corresponding cohomological class is $\{f^\xi\}=Obs(\xi)$

It is fair to illustrate the utility of $f$ now. The hindrance $h$ can be roughly viewed as a ``pre-obstruction" incarnating the difference in $R^\mu$, but it has several defects: first of all, $h$ belongs to the twisted module $C^2(A, K)$ of no cochain complex. Furthermore, $h=h(\mu)$ so it does not solely rely on the choice of the coulping $\xi$. As we have seen in this part, the obstruction cocycle $f$ nullify all the drawbacks. Indeed, this is the essence of this classical problem.

The following proposition resonates ones in \cite{LiMishGa14}. 
\begin{prop}
It can be summarized by these following digramms:
\[
\begin{tikzcd}
   &&& Inn(K) \ar[hook]{d}   \\
	  0 \ar{r} &AnniK \ar{r}{i} &K \ar{r}{\epsilon} &Mul(K) \ar{r}{\natural} &Out(K) \ar{r} &0 \\
	 &&& A\ar[dashed]{u}{\mu} 
	       \ar[swap]{ur}{\xi}    
\end{tikzcd}
\]

\[
\begin{tikzcd}
   &&& Inn(K) \ar[hook]{d} \ar[dashed, bend left]{dr}{\natural} \\
	 0 \ar{r} &AnniK \ar{r}{i} &K \ar{r} \ar[two heads]{ur}{\epsilon} &Mul(K) \ar{r} &Out(K) \ar{r} &0  \\
	 &&& A\otimes A \ar{u}
	                \ar[dashed]{ul}{h}
									\ar[bend right, crossing over, swap, near end, shift right=2.8ex]{uu}{R^\mu}
									\ar[bend right, swap, yshift=1ex]{ur}{R^\xi}
\end{tikzcd}
\]
where $\natural \circ R^\mu= R^\xi=0$

\[
\begin{tikzcd}
   0 \ar{r} &AnniK \ar{r}{i} &K \ar{r}{\epsilon} &Mul(K) \ar{r}{\natural} &Out(K) \ar{r} &0 \\
   &&& A\otimes A\otimes A \ar{ull}{\Delta^\mu h}
	                         \ar[swap]{u}{\Delta^\mu R^\mu}
\end{tikzcd}
\]
where $\Delta^{\mu} R^{\mu}=0$.
\end{prop}


\section{Identifying Lie and Associative Cochains}

In \cite{Hoch54b}, Hochschild formulates a proper cochain complex for computing the cohomology of both the ordinary and restricted Lie algebra. This equivalence, identifying Chevelley-Elienberg complex and the (normalized) Cartan-Hochschild standard complex in a particular way, is recalled in Tylor Evans' PhD thesis, \cite{Evan00}. We are going to generalize the definitions for computing Lie algeborid cohomology in terms of its universal enveloping algebroid.

Let $\mathfrak{g}$ be a Lie algebra over field $\mathbb{F}$ and let $\mathcal{M}$ be a $\mathfrak{g}$-module. It is well-known that there is a one-to-one correspondence between the Lie algebra representations of $\mathfrak{g}$ and the unitary representations of $U(\mathfrak{g})$. In this way, one may view $\mathcal{M}$ as a unitary $U(\mathfrak{g})$-module. 

1) Let us define the complex of ``Lie type":

$$\mathcal{C}_*=\{ \mathcal{C}_n, d^{\mathcal{C}} \}$$ 
where
$$\mathcal{C}_n :=U(\mathfrak{g}) \otimes \bigwedge\nolimits^{\!n} \mathfrak{g}$$

Clearly, each $\mathcal{C}_n $ becomes a $U(\mathfrak{g})$-module in a natural fashion.

The coboundary operator $d^{\mathcal{C}}_n: \mathcal{C}_n \rightarrow \mathcal{C}_{n-1} $ is defined by
\begin{align*}
d^{\mathcal{C}}_n(\mathfrak{u} \otimes x_1 \wedge \cdots \wedge x_n)&:=\sum_{i=1}^{n}(-1)^{i-1}\mathfrak{u} x_i \otimes x_1 \wedge \cdots \wedge \hat{x_i} \wedge \cdots \wedge x_n \\
& +\sum_{1 \leq s<t \leq n}(-1)^{s+t-1} \mathfrak{u} \otimes [x_s, x_t] \wedge x_1 \wedge \cdots \wedge \hat{x_s} \wedge \hat{x_t} \wedge \cdots x_n
\end{align*}

Consider the canonical augmentation $\epsilon: U(\mathfrak{g}) \rightarrow \mathbb{F}$ induced by the map $T(\mathfrak{g}) \mapsto \mathbb{F}$. Since the augmentation is surjective, we denote its kernel by
$$U(\mathfrak{g})_+=ker \epsilon,$$
that is, all of the positive parts of tensor algebra of $\mathfrak{g}$ passing over the quotient. $U(\mathfrak{g})_+$ will then play a key component of the tensor product $\tilde{S}(\cdot)$ which is the so called {\em normalized standard complex}. 


2)Let us now define the complex of ``associative type":
$$\mathcal{D}_*=\{ \mathcal{D}_n, d^{\mathcal{D}} \}$$ 
where
\begin{align*}
\mathcal{D}_n &:=U(\mathfrak{g}) \otimes {\underbrace{U(\mathfrak{g})_{+} \otimes \cdots \otimes U(\mathfrak{g})_{+}}_\text{$n$-times}} \\
              &= U(\mathfrak{g}) \otimes U(\mathfrak{g})_{+}^{\otimes^{n}}
\end{align*} 

The coboundary operator $d^{\mathcal{D}}_n: \mathcal{D}_n \rightarrow \mathcal{D}_{n-1} $ is defined by
\begin{align*}
d^{\mathcal{D}}_n(\mathfrak{u} \otimes x_1 \otimes \cdots \otimes x_n)&:=\mathfrak{u} x_1 \otimes \cdots \otimes x_n \\
&\sum_{i=1}^{n}(-1)^i \mathfrak{u} \otimes x_1 \otimes \cdots \otimes x_i x_{i+1} \otimes \cdots \otimes x_n
\end{align*}

In addition, we set $\mathcal{C}_0=\mathcal{D}_0=U(\mathfrak{g})$. 


In fact, we can show that the following two augmented complexs are free resolutions(acyclic) of $U(\mathfrak{g})$-modules:
\[\mathcal{C}_* \twoheadrightarrow \mathbb{F} \rightarrow 0\]
and  
\[\mathcal{D}_* \twoheadrightarrow \mathbb{F} \rightarrow 0 \]

\[
\begin{tikzcd}
\cdots \ar{r} 
&U(\mathfrak{g}) \otimes \bigwedge\nolimits^{\!2}\mathfrak{g} \ar{r}{d_2^\mathcal{C}} \ar{d}{\gamma} 
&U(\mathfrak{g}) \otimes \bigwedge\nolimits^{\!1}\mathfrak{g} \ar{r}{d_1^\mathcal{C}} \ar{d}{\gamma} 
&U(\mathfrak{g}) \ar[two heads]{r}{\epsilon} \ar{d}{\gamma = id} 
&\mathbb{F} \ar{r} \ar{d}{id_\mathbb{F}} 
&0     \\
\cdots \ar{r} 
&U(\mathfrak{g}) \otimes U(\mathfrak{g})_{+}^{\otimes^{2}} \ar[swap]{r}{d_2^\mathcal{D}}  
&U(\mathfrak{g}) \otimes U(\mathfrak{g})_{+}^{\otimes^{1}} \ar[swap]{r}{d_1^\mathcal{D}}  
&U(\mathfrak{g}) \ar[two heads, swap]{r}{\epsilon} 
&\mathbb{F} \ar{r}  
&0 
\end{tikzcd}
\]

If $\alpha$ and $\beta$ are any two of chain maps, we define the {\em chain homotopy} between these two complexs by assigning a family of operators $H_n: \mathcal{C}_n \rightarrow \mathcal{D}_{n+1}$ such that $d^{\mathcal{D}}_{n+1} \circ H_n + H_{n-1} \circ d^{\mathcal{C}}_n=\beta_n-\alpha_n$ for all $n$. Specifically, we require the definition:
\begin{align*}
H(\mathfrak{u})&=1 \otimes (\mathfrak{u}-\epsilon(\mathfrak{u})) \\
H(\mathfrak{u}\otimes x_1 \otimes \cdots \otimes x_n)&=1 \otimes (\mathfrak{u}-\epsilon(\mathfrak{u})) \otimes x_1 \otimes \cdots \otimes x_n
\end{align*}

Let $id=\gamma: \mathcal{C}_0 \rightarrow \mathcal{D}_0$. For $n>0$, define $\gamma:\mathcal{C}_* \rightarrow \mathcal{D}_*$ by
 
$$\gamma(\mathfrak{u} \otimes (x_1 \wedge \cdots \wedge x_n))=\sum_{\sigma} (sgn\sigma) \mathfrak{u}\otimes x_{\sigma(1)} \otimes \cdots \otimes x_{\sigma(n)}$$

Obviously, $\epsilon \circ id = \epsilon $. Next, $\gamma$ becomes an augmentation-preserving chain map once by justifying $\gamma \circ d_i^\mathcal{C} = d_i^\mathcal{D} \circ \gamma$. By interchanging the position of Lie and associative cochains, we can define a new map which actully serves as the inverse of $\gamma$ It follows that $\gamma$ is a chain homotopy equivalence. 

Let us turn our concern back to those resolutions for while. Chopping the $-1^{th}$ terms $\mathbb{F}$ and applying the left exact contravariant functor $Hom_{U(\mathfrak{g})}(-,\mathcal{M})$ to each term of both complexs, we have 
\[
\begin{tikzcd}
0 \ar{r} 
&\mathcal{M}^\mathfrak{g} \ar{r}{\delta_0^\mathcal{C}} \ar{d}
&Hom_{U(\mathfrak{g})}\big( U(\mathfrak{g}) \otimes \bigwedge\nolimits^{\!1} \mathfrak{g}, \mathcal{M} \big) \ar{r}{\delta_1^\mathcal{C}} \ar{d}
&Hom_{U(\mathfrak{g})}\big( U(\mathfrak{g}) \otimes \bigwedge\nolimits^{\!2} \mathfrak{g}, \mathcal{M} \big)  \ar{r}{\delta_2^\mathcal{C}} \ar{d} 
&\cdots             \\
0 \ar{r} 
&^{U(\mathfrak{g})} \mathcal{M}^{\textbf{0}} \ar[swap]{r}{\delta_0^\mathcal{D}}  
&Hom_{U(\mathfrak{g})} \big( U(\mathfrak{g}) \otimes U(\mathfrak{g})_{+}^{\otimes^{1}}, \mathcal{M} \big) \ar[swap]{r}{\delta_1^\mathcal{D}} 
&Hom_{U(\mathfrak{g})} \big( U(\mathfrak{g}) \otimes U(\mathfrak{g})_{+}^{\otimes^{2}}, \mathcal{M} \big) \ar[swap]{r} {\delta_2^\mathcal{D}} 
&\cdots
\end{tikzcd}
\]
where the zeroth term $Hom_{U(\mathfrak{g})} \big(U(\mathfrak{g}), \mathcal{M} \big)$ is determined by the derived functor of $\mathfrak{g}$-\textbf{Mod} or $U(\mathfrak{g})$-$\textbf{0}$-\textbf{Bimod} mapping into $R$-\textbf{Mod} respectively.

Furthermore, we have the following ``equivariant"(as in topology) isomorphisms for $n \geq 1$:
$$Hom_{U(\mathfrak{g})}\big( U(\mathfrak{g}) \otimes \bigwedge\nolimits^{\!n} \mathfrak{g}, \mathcal{M} \big) \cong Hom_{\mathbb{F}}\big(\bigwedge\nolimits^{\!n} \mathfrak{g}, \mathcal{M} \big), $$ and 

$$Hom_{U(\mathfrak{g})} \big( U(\mathfrak{g}) \otimes U(\mathfrak{g})_{+}^{\otimes^{n}}, \mathcal{M} \big) \cong Hom_{\mathbb{F}}\big(U(\mathfrak{g})_{+}^{\otimes^{n}}, \mathcal{M} \big).$$

Therefore, the corresponding coboundary operators $\delta_n^{\mathcal{C}}$ and $\delta_n^{\mathcal{D}}$ of vector spaces(instead of $U(\mathfrak{g})$-modules) are given in the usual sense. More precisely,
\begin{align*}
\delta_n^{\mathcal{C}}(f)&(x_1 \wedge \cdots \wedge x_{n+1})
:=\sum_{i=1}^{n} (-1)^{i} x_i f(x_1 \wedge \cdots \wedge \hat{x_i} \wedge \cdots \wedge x_{n+1})  \\
&+\sum_{1 \leq i<j \leq n+1} (-1)^{i+j-1} f([x_i,x_{i+1}] \wedge x_1 \wedge \cdots \wedge x_{n+1}),
\end{align*}
and
\begin{align*}
\delta_n^{\mathcal{D}}(f')&(x_1 \otimes \cdots \otimes x_{n+1})
:=x_1 f'(x_2 \otimes \cdots \otimes x_{n+1})  \\
&+\sum_{i=1}^{n}(-1)^{i}f'(x_1 \otimes \cdots \otimes x_i x_{i+1} \cdots \otimes x_{n+1}).
\end{align*}

Note that for the associative cochain, we made the RIGHT ACTION to be ZERO in the original definition of Hochschild differentials for associative algebra!

Thus, its Lie algebra cohomology with the differential considered above is
\begin{align*}
H^i_{CE}(\mathfrak{g}, \mathcal{M})&:=H^* \big(\mathcal{C}^n(\mathfrak{g}, \mathcal{M}) \big) \\
&=Ext_{U(\mathfrak{g})}^i (\mathbb{F}, \mathcal{M}),
\end{align*}
and the Hochschild cohomology
\begin{align*}
H^i_{Hoch}(U(\mathfrak{g}), \mathcal{M})&:=H^* \big(\mathcal{D}^n(U(\mathfrak{g})_+, \mathcal{M}) \big)   \\
&=Ext_{U(\mathfrak{g})-\textbf{0}}^i (\mathbb{F}, \mathcal{M}).
\end{align*}

For every associative $n$-cochain $f \in \mathcal{D}^n(U(\mathfrak{g})_+, \mathcal{M})$, we define the Lie cochain by
$$ f'(x_1 \wedge \cdots \wedge x_n):=\sum_{\sigma} sgn(\sigma) f(x_{\sigma(1)} \otimes \cdots \otimes x_{\sigma(n)}) $$

A direct computation shows $(\delta f)'=\delta(f')$ so that the map $f \mapsto f'$ induced a homomorphism on the cohomology groups. On the other hands, we should point out that $\gamma$ together with its inverse $\gamma^{-1}$ induce a isomorphism on the homology groups. Its dual map $\gamma^*$ is then identical to the map $f \mapsto f'$ and becomes a quasi-isomorphism as well. 

Consequently, the induced map
\[ 
\begin{tikzcd}
\gamma^*: H^*_{CE}(\mathfrak{g}, \mathcal{M}) \ar{r}{\cong} &H^*_{Hoch}(U(\mathfrak{g}), \mathcal{M}) 
\end{tikzcd}
\]
is an isomorphism for all $n$.


\section{Main Theorems}

In \textbf{Appendix C} one has our first correspondence
\begin{center}
An $A$-kernel $(K,\xi)$ is derived from an extension of algebras $\Leftrightarrow Obs(\xi)=0$ in $H^3(A, AnniK)$
\end{center}

Now we wish to show the second correspondence
\[
\begin{tikzcd}
\Big\{\big[(\xi, K)\big]_N \Big\} \ar{r}{Obs} &H^3(A, N)
\end{tikzcd}
\]
is a module isomorphism where the LHS is the vector space of equivalence class of $A$-kernels with common biannihilators $N$, and the RHS is the Hochschild cohomology group. Note that if we would like to clarify the representation that induces the differential for the cohomology, then we specify it by $H^3(A, N, \rho)$.The most difficult part is to show this map is an surjection. That is to say, given any cohomology class in $H^3$, we can construct a proper $A$-kernel whose derived obstruction is identical to the class. It actually describes the structure of the kernels.

Recall that $\xi$ defines an $A$-$A$-bimodule structure on $N$ in the midway of introducing the special cohomology and it does not rely on the choice of connections or bimultiplication laws that cover $\xi$. We sometimes refer this bimodule a \textbf{nucleus} of the kernel. In terms of the set of algebra kernels with common biannihilator, we can also say they have common nucleus. So nucleus-obstruction is the only twins determined by $\xi$ exclusively, while ``connection-hindrance" is otherwise. This just spells out its peculiarity.

To assign the set of $A$-kernels a linear structure, we define the addition and scalar multiplication as follows:

Define 
\[
K_1+K_2:=(K_1 \oplus K_2)/_{\{(n,-n)|n\in N\}}
\]

The factoring ideal is to cancel out those $k$ such that $x\cdot (k_1+k_2)=0$ resulting from $x=k+(-k)$ for some $x\in K_1+K_2$, and one has $Anni(K_1+K_2)=(N+N)/_{\{(n,-n)\}} \cong N$ immediately.

As $K_1+K_2$ is defined, for any two couplings $\xi_1$ and $\xi_2$, we find two (regular) bimultiplication laws $\mu_i: A \rightarrow Mul(K_i), i=1, 2$ that cover them. For any $a\in A, k_1\in K_1$ and $k_2\in K_2$, elements in the images $\mu_i(A)$ satisfy $\mu_i(a)(k_i)=(u_a, v_a)(k_i)=(a\cdot k_i, k_i \cdot a)$ for $i=1,2$.

Now $a$ acts on the direct sum of $K_i$ componentwise, after passing the quotient we have it on $[K_1\oplus K_2]$ so does on $K_1+K_2$. We denote it by $u_a([k])=a\cdot [k]$ and $v_a([k])=[k] \cdot a$ for any $[k]=[k_1\oplus k_2]\in K_1+K_2$. Thus we can define 
\[
\mu_1+\mu_2: A \rightarrow Mul(K_1+K_2):=End(K_1+K_2) \oplus End(K_1+K_2)^{op}
\]
where its image $(\mu_1+\mu_2)(A)$ consists of (permutable) elements with operations indicated as above.

Let $\xi_i: A \rightarrow Out(K_i)$ be two couplings, choose some $\mu_i$ that cover them respectively, then $\mu_1+\mu_2$ covers $\xi_1+\xi_2$. Initially, $(\xi_1+\xi_2)(A)$ consists of elements $\{([u_1+u_2]_a, [v_1+v_2]_a)\}$ derived from those elements in $\xi_1(A)$ and $\xi_2(A)$. This defines the sum of two couplings. Hence we have $(\xi_1+\xi_2, K_1+K_2)$.

Define 
\[
_{\lambda}K=(K \oplus N)/_{\{(k, -\lambda n)\}}
\]

Again, we have $Anni(_{\lambda}K)=(N\oplus N)/_{\{(n, -\lambda n)\}} \cong N$. This gives us $(_{\lambda}\xi, _{\lambda}K)$.

\begin{prop} ~\\ 
1) $Obs(\xi_1+\xi_2)=Obs(\xi_1)+Obs(\xi_2)$; \\
2) $Obs(_{\lambda}\xi)=_{\lambda}Obs(\xi)$; \\
3) Denote $(\cdot)^{ext}$ for an extendible kernel. If $(\xi_1, K_1)^{ext}, (\xi_2, K_2)^{ext}$, then $(\xi_1+\xi_2, K_1+K_2)^{ext}$; \\
4) If $_{\lambda}(\xi, K)^{ext}$, then $(_{\lambda}\xi, _{\lambda}K )^{ext}$.
\end{prop}

Given $\xi_i$ withe some coverings $\mu_i$, let us consider the following diagram:
\[
\begin{tikzcd}
& Mul(K_1) \ar{r} \ar[dashed]{dd}{\bar{\sigma}} & Out(K_1) \ar{dd} \\
A \ar{ur}{\mu_1} \ar[bend right= 5 ]{urr}{\xi_1} \ar[swap]{dr}{\mu_2} \ar[bend left= 5, swap]{drr}{\xi_2} \\
& Mul(K_2) \ar{r} & Out(K_2)
\end{tikzcd}
\]
where $Anni K_1=Anni K_2=N$, that is, the kernels $(\xi_1, K_1), (\xi_2, K_2)$ with common nucleus $N$ and $\bar{\sigma}$ is induced by $\sigma: K_1 \rightarrow K_1$ with $\sigma(N)=N$. 

\begin{defn}
Two $A$-kernels with common nucleus $N$ is said to be \textbf{isomorphic}, or $(\xi_1, K_1)_N \cong (\xi_2, K_2)$, if there is an isomorphism of algebras $\sigma$ fixing the biannihilator and $\bar{\sigma} \circ \mu_1=\mu_2$. 
\end{defn} 

The latter statement means $\bar{\sigma} (u_1, v_1)_a(k)=(u_2, u_2)_a(k)$ is a \emph{homomorphism} of images $\mu_i(A) \subset Mul(K_i)$ for any pairs of endomorphisms with respect to any $k$.

\begin{defn}
Two kernels with common nucleus $N$ are said to be \textbf{equivalent}, or $(\xi_1, K_1)_N \sim (\xi_2, K_2)_N$, if there are two extendible kernels $(\eta_1, S_1)^{ext}, (\eta_2, S_2)^{ext}$ with the \emph{same} nucleus such that
\[
(\xi_1+\eta_1, K_1+S_1) \cong (\xi_2+\eta_2, K_2+S_2)
\]
\end{defn}

Denote the equivalence class of kernels with common nucleus by 
\[
\big[(\xi, K)\big]_N:=(\xi, K)_N/_ \sim
\]

Given any algebra $A$, an $A$-$A$-bimodule $M$ and any element in $H^3(A, M)$, we would like to find a proper kernel (and a coupling in turn) that realized it. Using the language of connection and representation from part \textbf{3}, we obtain our first structure theorem: 

\begin{thm}
Given any associative algebra $A$ and let $(M, \rho)$ be an representation of $A$ where $\rho: A \rightarrow Mul(M)$ is a flat connection. Let $c$ be an element in $HH^3(A, M, \rho)$, then

1)there exists an algebra $K$ having a left $A$-module structure and such that $AnniK=M$,

2)there exists a homomorphism $\xi: A \rightarrow Out(K)$ such that the induced central representation $\rho^{\xi}: A \rightarrow Mul(AnniK)$ is equal to $\rho$, and

3) $Obs(\xi)$ coincides with $c$. 

Moreover, $(\xi, K)$ becomes the coupling of $A$.
\end{thm}
\begin{proof} ~\\
According to \cite{Hoch46} the proof is highly constructive. We simply sketch each step here. For details, consult \textbf{Appendix D}.

\begin{itemize} 
\item (Step 1) Define all the direct summands of $K$.

\item (Step 2) Define multiplications between the components of $K$.

\item (Step 3) Show the biannihilator of $L$ is trivial so that $AnniK=M$.

\item (Step 4) Define the left and right $A$-actions on $K$.

\item (Step 5) Show the four conditions hold whence we find a concrete connection $\bar{\mu}: A \rightarrow Mul(K)$.

\item (Step 6) Set $\bar{\xi}:= \natural \circ \bar{\mu}$. Hence $\bar{\mu}|_{AnniK}$ depends on the choice of $\bar{\xi}$ and we denote it by $\bar{\mu}|_{AnniK}=\rho^{\bar{\xi}}$. The pair $(M, \rho^{\bar{\xi}})$ becomes the central representation of $A$. It induces a differential $\delta^{\bar{\xi}}$ and actually, $\delta^{\bar{\xi}}=\delta^\rho$.

\item (Step 7) Write $c=\{f\}$ where the representative cocycle $f$ determines an \emph{extension of bimodules} of $M$ by $A\otimes A\otimes A^*$ within $K$. 

\item (Step 8) There is a suitable cochain $\bar{h}: A \otimes A \rightarrow E$ for some $E$ with $M\subset E\subset K$ such that $\Delta|_{E} \bar{h}=f$. Moreover, $\bar{h}$ becomes the hindrance of $\xi$ (see Lemma 6.1 and 6.2). Define a proper bilinear map
\[
\bar{h}(a_1 \otimes a_2)=a_1 \otimes a_2 \otimes 1
\]

\item (Step 9) Set $F^{\bar{\xi}}:=F(\bar{\mu}, \bar{h})$. Then $Obs(\bar{\xi})=\{F^{\bar{\xi}}\}$ and is identical to $c$. Namely, $F^\xi \equiv f$.
\end{itemize}
\end{proof}

What we really need later is an simplified version of above theorem. This refines the structure of extension of bimodules in step 7 by reducing $A\otimes A\otimes A^*$ to $A\otimes A$. We endow it with a bimodule structure through the following operations:
\begin{align*}
a_0 \cdot (a_1 \otimes a_2)&:=a_0 a_1 \otimes a_2 -a_0 \otimes a_1 a_2     \\
(a_1 \otimes a_2)\cdot a_0 &:=0
\end{align*}

The next two important lemmas are attributed to Hochschild:

\begin{lem}
Let $Q$ be any bimodule over $A$. Any element $f\in Z^3(A, Q)$ in the light of Hochschild cohomology defines a split extension of bimodules of $M$ by $A\otimes A$.
\end{lem}
\begin{proof}
Let the underlying vector space $E=A\otimes A \oplus Q$ with a bimodule structure defined as follows
\begin{align*}
a \cdot (p, q)&:= (a\cdot p, f(a\otimes a_1 \otimes a_2)+a \cdot q)  \\
(p,q)\cdot a&:=(p\cdot a, q\cdot a)
\end{align*}
Define $\pi: E \rightarrow A\otimes A$ by $\pi(p, q)=p$. We claim that $(E, \pi)$ becomes a split extension of bimodules. Indeed, $ker \pi=\{(0,q)|\pi(0,q)=0, \forall q \in Q \}$, so we can identify the sub(bi)module $(0, Q)$ with $Q$. 
\[
\begin{tikzcd}
0 \ar{r} &Q \ar[tail]{r} &E \ar{r}{\pi} &A\otimes A \ar{r} \ar[dashed, bend left, xshift=-1mm]{l}{\gamma} &0
\end{tikzcd}
\]
Set $\gamma: A \otimes A \rightarrow E$ with 
\[
\gamma(p):=(p,0)
\]
Then we have $\pi \gamma=id_{A\otimes A}$ for $\pi \circ \gamma(p)=\pi(p,0)=p$, as desired. Finally, we define 
$$\varphi_a(p):=a \cdot \gamma(p)- \gamma(a\cdot p)$$
Since $\pi \varphi=a\cdot p- a\cdot p=0$ then $\varphi_a(p) \in Q$, and we have $\varphi_a: A\otimes A \rightarrow Q$. Therefore the map $a \mapsto \varphi_a$ defines an element 
\[
f^\gamma \in Hom \big(A, Hom(A\otimes A, Q) \big) \cong Hom(A\otimes A \otimes A, Q).
\]
One can check that $\delta_{Hom(A\otimes A, Q)} f^\gamma=0$ with a proper bimodule structure on $Hom(A\otimes A, Q)$. Hence, $f^\gamma$ becomes a cocycle.

On the other hand, we have $\varphi_a(p)=a \cdot (p,0)- (a\cdot p, 0)=(a\cdot p, f(a\otimes a_1 \otimes a_2))-(a\cdot p,0)=(0, f(a\otimes a_1 \otimes a_2))$. We conclude that $f^\gamma \equiv f$ and the lemma is proved. 
\end{proof}

\begin{lem}
Given $f\in Z^3(A, Q)$ with the corresponding induced split extension of bimodules as above. There is an element $h_E \in C^2(A, E)$ such that $\delta_E h_E=f$.
\end{lem}

\begin{proof}
Define $h: A\otimes A \rightarrow E$ by
\[
h_E(a_1 \otimes a_2):=(a_1 \otimes a_2,0)
\]

Next, we compute
\begin{align*}
\delta_E h_E(a \otimes a_1 \otimes a_2)&=a\cdot h_E(a_1 \otimes a_2) -h_E(a a_1 \otimes a_2)+h_E(a\otimes a_1 a_2)-h(a\otimes a_1) \cdot a_2 \\
& =a \big(a_1 \otimes a_2, 0 \big)-\big(a a_1 \otimes a_2, 0 \big)+ \big(a\otimes a_1 a_2, 0 \big)   \\
& =\big(a a_1 \otimes a_2-a\otimes a_1a_2, f(a\otimes a_1 \otimes a_2)\big)-\big(a a_1 \otimes a_2, 0 \big)+ \big(a\otimes a_1 a_2, 0 \big) \\
&=\big(0, f(a\otimes a_1 \otimes a_2)\big) \in (0,Q)
\end{align*}

As we have identified $(0,Q)$ with $Q$, then $f$ becomes a coboundary.
\end{proof}

Since we will use the enveloping algebra of Lie algebra $\mathfrak{g}$, we appropriate the position of $A$ by denoting $U=U(\mathfrak{g})$ in next theorem:

\begin{thm}
If in particular the right $U$-action on $P_2=U \otimes U$ is trivial, then

a) reset $h(a_1 \otimes a_2)=a_1 \otimes a_2$ such that
$\epsilon \circ h = R^\mu$, and by lemma 6.2 we have
$\Delta|_E h \equiv f$;   
  
b) in this case the structure of $K$ can be simplified.
\end{thm}
\begin{itemize}
\item See \textbf{Appendix E} for its proof. Note that this is the pivotal bridge theorem where we are forwarding to the Lie algebra case.
\end{itemize}

\section{Going to Lie}
One of Shukla's unproved theorem in \cite{Shk66} states the generalized version in terms of DG-Lie algebra. The structure theorem for any ordinary Lie algebra kernels in \cite{Hoch56a} is laconic and thus least readable, so we will clarify his dense writing and present a formal proof in the following paragraphs.

Let us formulate our second main theorem at first:
\begin{thm}
Given any Lie algebra $\mathfrak{g}$ and let $\mathcal{M}$ be any $\mathfrak{g}$-module. Let $\rho_{Lie}: \mathfrak{g} \rightarrow Der(\mathcal{M})$ be a The Lie algebra homomorphism (i.e. a flat $\mathfrak{g}$-connection on $\mathcal{M}$). For any element $\{f\}$ in $H^3(\mathfrak{g}, \mathcal{M}, \rho_{Lie})$, 

1)there exists a Lie algebra $\mathfrak{K}$ having some module structures over $\mathfrak{g}$ and such that $Z\mathfrak{K}=\mathcal{M}$,

2)there exist a homomorphism $\Xi: \mathfrak{g} \rightarrow Out(\mathfrak{K})$ such that the induced central representation $\rho^{\Xi}: \mathfrak{g} \rightarrow Der(Z\mathfrak{K})$ coincides with $\rho_{Lie}$, and

3) $Obs(\Xi)=f$ 
 
Moreover, $(\Xi, \mathfrak{K})$ becomes a (Lie) coupling of $\mathfrak{g}$.
\end{thm}

\begin{proof}
We are actually beginning with the given Lie triple: 
$$\big(\mathfrak{g}, \mathcal{M}, f \big)$$
where $\in Z^3(\mathfrak{g}, \mathcal{M}, \rho_{Lie})$.
The main technique is to transfer the Lie triple into some proper associative triples and then to go back to Lie by manipulating formulas. First of all, we take $U(\mathfrak{g})$. 

The second important note is the equivalence of categories
\begin{center}
$\textbf{Rep}(\mathfrak{g}) \Leftrightarrow U(\mathfrak{g})$-$\textbf{Mod}$
\end{center}

Let $M$ be a $U(\mathfrak{g})$-module corresponding to the given representation pair $(\mathcal{M}, \rho_{Lie})$.

Thirdly, we will the previous section to get 
\begin{align*}
Hom_{U(\mathfrak{g})} \big(U(\mathfrak{g}) \otimes \bigwedge\nolimits^{\!3} \mathfrak{g}, \mathcal{M}) \big) 
&\rightarrow 
Hom_{U(\mathfrak{g})} \big(U(\mathfrak{g})\otimes U(\mathfrak{g})_+^{\otimes^3}, M) \big) \\
f &\mapsto f'
\end{align*}
such that the induced cohomology groups are isomorphic, that is $$H^3(\mathfrak{g}, \mathcal{M}, \rho_{Lie}) \cong HH^3(U(\mathfrak{g}), M, \rho)$$
where 
\[
\begin{tikzcd}
\mathfrak{g} \ar{r}{\rho_{Lie}} \ar[swap]{d}{\epsilon} &Der(\mathcal{M}) \ar{d} \\
U(\mathfrak{g}) \ar[swap]{r}{\rho} &Mul(M) 
\end{tikzcd}
\]
Therefore, our associative triple is
$$\big(U(\mathfrak{g}), M, f'\big)$$
where $f'\in Z^3(U(\mathfrak{g}), M, \rho)$

Due to the construction of Theorem $1$, our $K$ will be a special combination of $A=U(\mathfrak{g})$ such that 
$$M=ZK:=\{m|m\cdot K=K \cdot m=0\}.$$ 
with the listed multiplications between all possible components of $K$ and $A$-actions. The coupling $\xi$ is therefore fulfilled by this $A$-action. Write 
\begin{align*}
\mu: U(\mathfrak{g}) &\rightarrow Mul(K)  \\
a &\mapsto (u_a,v_a)
\end{align*}
Then $\xi=\natural \circ \mu$ for some proper linear mappings $\theta$ and the following induced central representation(only determined by the given coupling) $\rho^\xi: U(\mathfrak{g}) \rightarrow Mul(ZK)$ which coincides with $\rho$ by Theorem $1$.

Denote $\mathfrak{K}=Lie(K)$ with $[k_1, k_2]=k_1 \cdot k_2-k_1 \cdot k_2$ such that
$$\mathcal{M}=Z\mathfrak{K}:=\{m|[m, K]=0 \}.$$

In the associative case, we know that the nonzero difference $R^\mu(a_1 \otimes  a_2)=\mu_{a_1}\mu_{a_2}-\mu_{a_1 a_2}$ is an inner bimultiplication effected by some bilinear maps $h:U(\mathfrak{g}) \otimes U(\mathfrak{g}) \rightarrow K$. As $h$ has already built in Theorem $1$, we want to define its counterpart $H$ for Lie algebra. More precisely, beginning with the covering-hindrance pair
$$(\mu, h) \quad \text{with} \quad \epsilon \circ h = R^\mu$$
we want to define the Lie-pair
$$(\nabla, H) \quad \text{with} \quad ad \circ H = R^\nabla$$

For any $\mathfrak{a} \in \mathfrak{g}$, we set 
\begin{align*}
\nabla:\mathfrak{g} & \rightarrow Der(\mathfrak{K}) \\
\mathfrak{a} & \mapsto  u_{\mathfrak{a}}-v_{\mathfrak{a}} 
\end{align*}
Recall that $\epsilon \circ h=(u_{h(a_1 \otimes a_2)}, v_{h(a_1 \otimes a_2)})$ is an inner bimultiplication produced by $h(a_1 \otimes a_2)$, so for any $k \in K$ we get
\begin{align*}
\epsilon \circ h (k) &= R^\mu (K)      \\
(u_{h(a_1 \otimes a_2)}, v_{h(a_1 \otimes a_2)})(k)
&=(\mu_{a_1}\mu_{a_2}-\mu_{a_1 a_2})(k)  \\
\Big(h(a_1\otimes a_2)k, kh(a_1\otimes a_2)\Big)
&=\Big( a_1\cdot(a_2\cdot k)-(a_1a_2)\cdot k, (k\cdot a_1)\cdot a_2-k\cdot(a_1a_2) \Big)
\end{align*}
\[
\begin{tikzcd}
U(\mathfrak{g})\otimes U(\mathfrak{g}) \ar{r}{h} 
& K \ar{d}{i} \\
\mathfrak{g} \wedge \mathfrak{g} \ar{u}{\gamma} \ar[swap, dashed]{r}{H} 
&\mathfrak{K}
\end{tikzcd}
\]
Then $H:=i \circ h \circ \gamma$, where $\gamma(\mathfrak{a_1} \wedge \mathfrak{a_2})=a_1 \otimes  a_2 -a_2 \otimes a_1$.
\[
\begin{split}
\Delta^\nabla H(\mathfrak{a_1}  \wedge \mathfrak{a_2} \wedge \mathfrak{a_3})
&=\Delta^\nabla \circ i \circ h \circ \gamma (\mathfrak{a_1}  \wedge \mathfrak{a_2} \wedge \mathfrak{a_3})  \\
&={i^*}\Delta^\nabla \circ h \Big(\sum_{\sigma}(sgn\sigma) a_{\sigma_{(1)}} \otimes a_{\sigma_{(2)}} \otimes a_{\sigma_{(3)}} \Big) \\
&=\Delta^\mu h \Big(\mathfrak{S}\{a_1 \otimes a_2 \otimes a_3 \} \Big) \\
&=\Big(0, \mathfrak{S}\{f'(a_1 \otimes a_2 \otimes a_3) \} \Big)\\
&=\Big(0, f(\mathfrak{a_1}  \wedge \mathfrak{a_2} \wedge \mathfrak{a_3}) \Big)
\end{split}
\]

The curvature map for Lie algebra 
$$R^\nabla: \mathfrak{g} \wedge \mathfrak{g} \rightarrow ad(\mathfrak{K})$$
is given by $R^\nabla(\mathfrak{a_1} \wedge \mathfrak{a_2})=[\nabla_\mathfrak{a_1}, \nabla_\mathfrak{a_2}]-\nabla_{[\mathfrak{a_1},\mathfrak{a_2}]}$. 

On the other hands, we have $R^\nabla=ad \circ H$. In details, 
\[
\begin{split}
R^\nabla(\mathfrak{a_1} \wedge \mathfrak{a_2})
&=ad \circ i \circ h \circ \gamma(\mathfrak{a_1} \wedge \mathfrak{a_2}) \\
&=ad \circ i \circ h(a_1 \otimes a_2 -a_2 \otimes a_1) \\
\end{split}
\]
For any $\mathfrak{k} \in \mathfrak{K}$,
\[
\begin{split}
ad_{i \circ h(a_1 \otimes a_2 -a_2 \otimes a_1)}(\mathfrak{k}) 
&=[i \circ h(a_1 \otimes a_2 -a_2 \otimes a_1), \mathfrak{k}] \\
&=[i \circ h(a_1 \otimes a_2), \mathfrak{k}]-[i \circ h(a_2 \otimes a_1), \mathfrak{k}] \\ 
&=\Big(h(a_1 \otimes a_2) k-k h(a_1 \otimes a_2) \Big)-\Big(h(a_2 \otimes a_1) k-k h(a_2 \otimes a_1) \Big)  \\
\text{in addition,}
&=(u_{h(a_1 \otimes a_2)}-v_{h(a_1 \otimes a_2)})-(u_{h(a_2 \otimes a_1)}-v_{h(a_2 \otimes a_1)})(k) \\
&=(u_{h(a_1 \otimes a_2)-h(a_2 \otimes a_1)},v_{h(a_1 \otimes a_2)-h(a_2 \otimes a_1)})(k) \\
&=(u_{H(\mathfrak{a_1} \wedge \mathfrak{a_1})},v_{H(\mathfrak{a_1} \wedge \mathfrak{a_1})})(k) \\
\end{split}
\]


By the definition of $\nabla$, we have
\begin{align*}
\nabla_\mathfrak{\mathfrak{a_1}} \nabla_\mathfrak{\mathfrak{a_2}}(\mathfrak{k})
&=(u_{\mathfrak{a_1}}-v_{\mathfrak{a_1}})(u_{\mathfrak{a_2}}-v_{\mathfrak{a_2}})(\mathfrak{k}) \\
&=u_{\mathfrak{a_1}}u_{\mathfrak{a_2}}-\mathbf{u_{\mathfrak{a_1}}v_{\mathfrak{a_2}}}-\mathbf{v_{\mathfrak{a_1}}u_{\mathfrak{a_2}}}+v_{\mathfrak{a_1}}v_{\mathfrak{a_2}}  (\mathfrak{k})  \\
\nabla_\mathfrak{\mathfrak{a_2}} \nabla_\mathfrak{\mathfrak{a_1}}(\mathfrak{k}) 
&=(u_{\mathfrak{a_2}}-v_{\mathfrak{a_2}})(u_{\mathfrak{a_1}}-v_{\mathfrak{a_1}})(\mathfrak{k})   \\
&=u_{\mathfrak{a_2}}u_{\mathfrak{a_1}}-\mathbf{u_{\mathfrak{a_2}}v_{\mathfrak{a_1}}}-\mathbf{v_{\mathfrak{a_2}}u_{\mathfrak{a_1}}}+v_{\mathfrak{a_2}}v_{\mathfrak{a_1}}  (\mathfrak{k})  \\
\nabla_{[\mathfrak{\mathfrak{a_1}},\mathfrak{\mathfrak{a_2}}]} (\mathfrak{k})
&=u_{[\mathfrak{a_1},\mathfrak{a_2}]}-v_{[\mathfrak{a_1},\mathfrak{a_2}]} (\mathfrak{k})  \\
&=u_{\mathfrak{a_1}\otimes \mathfrak{a_2}-\mathfrak{a_2} \otimes \mathfrak{a_1}}-v_{\mathfrak{a_1}\otimes \mathfrak{a_2}-\mathfrak{a_2} \otimes \mathfrak{a_1}} (\mathfrak{k})    \\
&=u_{\mathfrak{a_1} \mathfrak{a_2}}-u_{\mathfrak{a_2} \mathfrak{a_1}}-v_{\mathfrak{a_1} \mathfrak{a_2}}+v_{\mathfrak{a_2} \mathfrak{a_1}} (\mathfrak{k}) 
\end{align*}

The last two lines hold for referring $[x \otimes y]-x\otimes y-y \otimes x$ in the canonical ideal and then abuse the tensor notation. When grouping these three nablas, the middle four bold parts are cancelled because of the permubalility condition $(a_1 \cdot k)\cdot a_2=(a_2 \cdot k)\cdot a_1$ and vice versa. Equivalently, $u_{a_1}v_{a_2}=v_{a_2}u_{a_1}$ and $u_{a_2}v_{a_1}=v_{a_1}u_{a_2}$. 

Therefore, we have
\[
\begin{split}
R^\nabla(\mathfrak{\mathfrak{a_1}} \wedge \mathfrak{\mathfrak{a_2}})(\mathfrak{k}) 
&= \big( \nabla_\mathfrak{\mathfrak{a_1}} \nabla_\mathfrak{\mathfrak{a_2}}-\nabla_\mathfrak{\mathfrak{a_2}} \nabla_\mathfrak{\mathfrak{a_1}}-\nabla_{[\mathfrak{\mathfrak{a_1}},\mathfrak{\mathfrak{a_2}}]} \big)(\mathfrak{k}) \\ 
&=u_{\mathfrak{a_1}}u_{\mathfrak{a_2}}+v_{\mathfrak{a_1}}v_{\mathfrak{a_2}}-u_{\mathfrak{a_2}}u_{\mathfrak{a_1}}-v_{\mathfrak{a_2}}v_{\mathfrak{a_1}}-u_{\mathfrak{a_1} \mathfrak{a_2}}+u_{\mathfrak{a_2} \mathfrak{a_1}}+v_{\mathfrak{a_1} \mathfrak{a_2}}-v_{\mathfrak{a_2} \mathfrak{a_1}} (\mathfrak{k})\\ 
&=\Big(u_{\mathfrak{a_1}}u_{\mathfrak{a_2}}-u_{\mathfrak{a_1} \mathfrak{a_2}}-(v_{\mathfrak{a_2}}v_{\mathfrak{a_1}}+v_{\mathfrak{a_1} \mathfrak{a_2}}) \Big)-\Big(u_{\mathfrak{a_2}}u_{\mathfrak{a_1}}-u_{\mathfrak{a_2} \mathfrak{a_1}}-(v_{\mathfrak{a_1}}v_{\mathfrak{a_2}}+v_{\mathfrak{a_2} \mathfrak{a_1}}) \Big) (\mathfrak{k}) \\ 
&=\Big(h(\mathfrak{a_1} \otimes \mathfrak{a_2}) \mathfrak{k}-\mathfrak{k} h(\mathfrak{a_1} \otimes \mathfrak{a_2}) \Big)-\Big(h(\mathfrak{a_2} \otimes \mathfrak{a_1}) \mathfrak{k}-\mathfrak{k} h(\mathfrak{a_2} \otimes \mathfrak{a_1}) \Big) 
\end{split}
\]

As two results coincide, we conclude that based on our definition of $\nabla$ and $R^\nabla$, we have found a proper $H$ derived from its associative counterpart $h$ such that $R^\nabla=ad \circ H$ as desired.
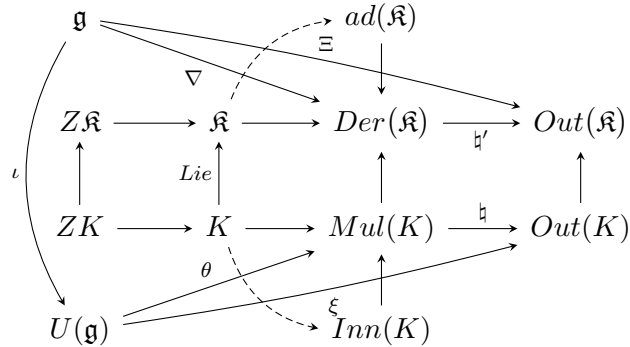
\[
\begin{tikzcd}
\mathfrak{g} \ar[bend right, swap]{ddd}{\iota} \ar[swap]{drr}{\nabla} \ar[bend left=3]{drrr}{\Xi}   
&& ad(\mathfrak{K}) \ar{d}  \\
Z\mathfrak{K} \ar{r} 
&\mathfrak{K} \ar{r} \ar[bend left, dashed]{ur}
&Der \big(\mathfrak{K} \big) \ar[swap]{r}{\natural'} 
&Out \big(\mathfrak{K} \big) \\
ZK \ar{r} \ar{u} 
&K \ar{r} \ar{u}{Lie} \ar[bend right, dashed]{dr}
&Mul(K) \ar{r}{\natural} \ar{u} 
&Out(K) \ar{u} \\
U(\mathfrak{g}) \ar[bend right=3, swap]{urrr}{\xi} \ar{urr}{\theta}
&& Inn(K) \ar{u}
\end{tikzcd}
\]
For some proper $\nabla$, our Lie coupling $\Xi$ can be viewed as the composition $\natural \circ \nabla$. We set $F^\Xi:=F(\nabla, H)$ then by the argument above we have $F(\nabla, H) = F(\mu, h)$, while the latter coincides with the \textit{a priori} associative cocycle $f'$ by Theorem $1$. Again $f' \mapsto f$ is an isomorphism onto the Lie cocycle. Consequently, following with all of the equalities, we get
$$F^\Xi=f$$ 
\[
\begin{tikzcd}
F(\nabla, H) \ar{r}{\text{above}} \ar[dashed]{d} &F(\mu, h) \ar{d}{\text{Theorem 1}} \\
f & f' \ar{l}{\text{Lie-Asso id.}} 
\end{tikzcd}
\]
\[
\begin{tikzcd}
&\mathfrak{g} \ar[bend right, swap]{ddd}{\iota} \ar[swap]{drr}{\nabla} \ar{drrr}{\Xi} \ar[swap]{dl}{\rho^{\Xi}} \\
Der(\mathcal{M}) 
&\mathcal{M} \ar{r} \ar[dashed]{l} 
&\mathfrak{K} \ar[swap]{r} 
&Der \big(\mathfrak{K} \big) \ar[swap]{r}{\natural} & Out \big(\mathfrak{K} \big)           \\
Mul(M) \ar{u} 
& M \ar{u} \ar{r} \ar[dashed]{l} 
&K \ar[swap]{r} \ar{u} 
&Mul(K) \ar{r}{\natural} \ar{u} 
&Out(K) \ar{u}                \\
&U(\mathfrak{g}) \ar[swap]{urrr}{\xi} \ar{urr}{\theta} \ar{ul}{\rho^{\xi}}
\end{tikzcd}
\]
\end{proof}

\section*{Appendix A: Generalities on Associative Algebras}
\addcontentsline{toc}{section}{Appendix A: Generalities on Associative Algebras}

The following couple of appendices aim to give a glossary about associative algebra and Hochschild cohomology compatible with this paper. A lot of its classical knowledge occurs in \cite{CE56}, \cite{Redo01} and in even those very old \cite{CG} and \cite{CT}.

Let $A$ be an associative algebra over any unital commutative ring $R$. The tensor product of two $R$-algebras $A$ and $B$ is $A \otimes_R B$ with an associative multiplication defined by $(a_1 \otimes b_1)(a_2 \otimes b_2)=(a_1 a_2)\otimes (b_1 b_2)$. A $R$-algebra homomorphism is both ring homomorphism and module homomorphism over $R$.  

Any $A$-$A$-bimodule $M$ is also a $R$-module in the following way: for any $r \in R$, since $(ra)\cdot m \in A$, then $(r_1 a_1) \big((r_2 a_2)\cdot m \big)=\big((r_1 a_1)(r_2 a_2) \big) \cdot m = (r_1 r_2) \big((a_1 a_2) \cdot m \big) \in A$; and similar for the right operation. Moreover, we can define a left $A \otimes_R B$-module in the following way: for any $a \in A$ and $b\in B$, let $(a \otimes b)\cdot m=a(b\cdot m)=b(a \cdot m)$. 

If $M$ is a $A$-$B$-bimodule, then it can be seen as a left $A \otimes_R B^{op}$-module in the following way: $(a \otimes b^*)m=a(b^*m)= b^*(am)$, where the opposite operation is given by $a^* m := m a$. Applying the star-operation, we just get the usual bimodule condition $a(mb)=(am)b$. In this way, it is possible to identify any $A$-$A$-bimodule with the left $A \otimes_R A^{op}$-module. Write $A^e := A \otimes_R A^{op}$ to be the evenloping algebra of $A$. This is again a $R$-algebra with the same multiplication defined as above. Lemma 2.1 from [Red00] tells us that the category of $A$-$A$-\textbf{bimodule} is equivalent to the category of left $A^e$-\textbf{module}. Therefore, our bimodule $M$ becomes a left $A^e$-module(as well as a $R$-module).

Let us recall the classical chain complex heading to the Hochschild cohomology. 

For all $n \geq 0$, let $S_n(A)=A \otimes_R A^{\otimes^n} \otimes_{R} A$ be the $(n+2)$-folds tensor product over $R$ of $R$-algebra $A$. It is an $A$-$A$-bimodule in a natural way, so it can be seen as an $A^e$-bimodule. The map $b_n: S_n(A) \rightarrow S_{n-1}(A)$ defined by
\[
b_n(a_1 \otimes \cdots \otimes a_n)=\sum_{i=0}^{n} (-1)^i a_1 \otimes \cdots \otimes a_i a_{i+1} \otimes \cdots \otimes a_n
\]
is an $A$-$A$-bimodule morphism and thus an $A^e$-module morphism. Let $S_{-1}(A)=A$ and let $b_0=\epsilon: S_0(A) \rightarrow S_{-1}(A)$ such that $\epsilon(a \otimes a')=aa'$ be the augmentation(also bimodule morphism). Then
\[
\begin{tikzcd}
0 & S_{-1}(A) \ar[swap]{l}
  & S_0(A) \ar[swap]{l}{\epsilon}  
	& S_1(A) \ar[swap]{l}{b_1} 
	& S_2(A) \ar[swap]{l}{b_2}  
	& \cdots \ar[swap]{l}{b_3}
\end{tikzcd}
\]
\[
\begin{tikzcd}
0 & A \ar[swap]{l}
  & A \otimes_R A \ar[swap]{l}{\epsilon}  
	& A \otimes_R A \otimes_R A \ar[swap]{l}{b_1} 
	& A \otimes_R A^{\otimes^2} \otimes_R A \ar[swap]{l}{b_2}  
	& \cdots \ar[swap]{l}{b_3}
\end{tikzcd}
\]
forms an acyclic complex. Indeed, consider the map $s: S_n(A) \rightarrow S_{n+1}(A)$ with $s(x)=1 \otimes x$. One can easily check that $b_n s +s b_{n-1}=id_{A^{\otimes^{n+1}}}$ and $b_0 s=id_A$. Additionally, we have $b^2=0$. If $A$ is $R$-projective, then $A^{\otimes^n}$ is also $R$-projective and $S_n(A)$ becomes $A^e$-projective. The projective resolution $(S(A), b)$ in above sense is called the \textbf{\textit{standard complex}}  or \textbf{\textit{bar resolution}} of $A$. 

Next, by chopping off the first nonzero term and applying the contravariant functor $Hom_{A^e}(-, M)=Hom_{A\otimes_{R} A^{op}}(-,M)$ to the chain resolution, we reach a cochain complex
\[
\begin{tikzcd}
0 \ar{r} & Hom_{A^e}\big(S_0(A), M \big) \ar{r}
         & Hom_{A^e}\big(S_1(A), M \big) \ar{r} 
	       & Hom_{A^e}\big(S_2(A), M \big) \ar{r} & \cdots
\end{tikzcd}
\]
of mere left-exactness. Now consider the following form 
$$S_n(A)=A \otimes_R A^{\otimes^n} \otimes_R A 
       \cong A \otimes_R \tilde{S}_n(A) \otimes_R A 
			 \cong A^e \otimes_R \tilde{S}_n(A),$$ 
where $\tilde{S_n}(A)$ is the $n$-folds tensor product of $A$ for all $n \geq 1$ and put $\tilde{S_0}(A)=R$. Then the hom functor gives  
$$ Hom_{A^e} \big(S_n(A), M \big) 
   \cong Hom_{A^e} \big(A^e \otimes_R \tilde{S}_n(A), M \big)
	 \cong Hom_R(\tilde{S_n}(A), M). $$
Namely,
\[
\begin{tikzcd} 
\cdots \ar{r} & Hom_{A^e}\big(S_n(A), M \big) \ar{r}{b_{n-1}^*} \ar{d}{\cong,\varphi} 
              & Hom_{A^e}\big(S_{n+1}(A), M \big) \ar{r} \ar{d}{\cong, \varphi} 
							& \cdots \\
\cdots \ar{r} & Hom_R \big(\tilde{S}_n(A), M \big) \ar[swap]{r}{\delta_n} 
              & Hom_R \big(\tilde{S}_{n+1}(A), M \big) \ar{r} 
							& \cdots \\
\end{tikzcd}
\]
where $\varphi: f \mapsto \tilde{f}$ with $f(x)=\tilde{f}(1 \otimes x \otimes 1)$ and $b^*\circ f=f\circ b$. In particular, the first few entries are:
\[
\begin{tikzcd} 
0 \rar &Hom_{A^e}(A\otimes_R A, M) \rar \dar 
       &Hom_{A^e}(A\otimes_R A \otimes_R A, M) \rar  \dar
       &\cdots  \\
0 \rar &Hom_{R}(R, M) \rar 
       &Hom_{R}(A, M) \rar						
       &\cdots  
\end{tikzcd}
\]
One can check that the above diagram is commutative and actually define the formula of $\delta_n: Hom_R \big(\tilde{S}_n(A), M \big) \rightarrow Hom_R \big(\tilde{S}_{n+1}(A), M \big)$ for each $n$. Generally, the coboundary operator is 
\begin{align*}
\delta_n(f)&(a_1 \otimes \cdots \otimes a_{n+1})
:=a_1 \cdot f(a_2 \otimes \cdots \otimes a_{n+1})  \\
&+\sum_{i=1}^{n}(-1)^{i}f(a_1 \otimes \cdots \otimes a_i a_{i+1} \otimes \cdots \otimes a_{n+1})+(-1)^{n+1}f(a_1 \otimes \cdots \otimes a_n) \cdot a_n
\end{align*}

Define the \textbf{\textit{$i^{th}$-Hochschild cohomology}} of $A$ with coefficients in an $A$-$A$-bimodule (here only consider its $R$-module structure) $M$:
$$HH^n(A, M):=H^* \big( Hom_R (\tilde{S}_n(A), M) \big)  $$
When using the projective resolution, we can make an alternative definition with Ext functor involved:
$$ HH^n(A, M):=Ext_{A^e}^i(A, M) $$
One may read [Dowdy69] and \cite{Car-E} for more detailed construction.

An elegant treatment on the interchange of these two types of cochains is the derived functor approach. 

\section*{Appendix B: Derived Functor Approach}
\addcontentsline{toc}{section}{Appendix B: Derived Functor Approach}

Given any $R$-algebra $A$ and let $M$ be an $A$-$A$-bimodule. In the spirit of part \textbf{3}, a \textbf{\textit{representation}} of $A$ on $M$ is a pair $(M, \rho)$ where $\rho: A \rightarrow Mul(M) $ is a $R$-algebra homomorphism. The set of representations of $A$ forms a category, $\mathbf{Rep}(A)$. For any $(M, \rho)$, we define an invariant sub(bi)module

$$M^{A^e}:= \{m \in M| {\rho_a} m-m {\rho_a}, \forall a\in A \} $$

Generally, this defines a functor:
\begin{center} 
$(-)^{A^e}: \textbf{Rep}(A)$ $\rightarrow R$-$\textbf{Mod},$  
\end{center} 
alternatively, we can express it as 
\begin{center} 
$^{A}(-)^{A}: A$-$A$-$\textbf{Bimod} \rightarrow R$-$\textbf{Mod}$
\end{center}  
such that, when writing $\rho_a m=am$ and $m \rho_a=ma$, 
$$ ^{A}M^{A}=\{ m|am-ma, \forall a\in A \},$$
for \textit{any} bimodule $M$. 

\textbf{The standard complex of $A$ with coefficients in a representation $(M, \rho)$} is $Hom_R(\tilde{S}_n(A), M)$ together with a $R$-linear differential $\delta_n^{\rho}: Hom_R(\tilde{S}_n(A), M) \rightarrow Hom_R(\tilde{S}_{n+1}(A), M)$ given by
\begin{align*}
\delta^\rho_n(f)&(a_1\otimes \cdots\otimes a_{n+1})
=u_{a_1}f(a_2\otimes \cdots \otimes a_{n+1})  \\
&+\sum_{i=1}^n (-1)^i f(a_1 \otimes \cdots \otimes a_i a_{i+1}\otimes \cdots \otimes a_{n+1}) +(-1)^{n+1}v_{a_{n+1}}f(a_1\otimes \cdots \otimes a_n),
\end{align*}
for any $a \in A$ and by indicating $(\rho_a m, m\rho_a)=(u_a m, v_a m)$. Write $\delta=\delta^\rho$ for short. 

We can easily verify that the functor treated above actually pushes the original Hochschild cochains forward to:
$$^{A}{\big( Hom_R(\tilde{S}_n(A), M) \big)}^{A}=Hom_R(S_n(A), M) $$
for all $n \geq 0$
Therefore, we have the \textbf{$i^{th}$-Hochschild cohomology group with coeffcients in a representation} to be the right derived functor:
$$HH^i(A, M, \rho):=\textbf{\textit{R}} {^{A}(-)^{A}}.$$
We shall be always highlighting the involved representations in the classical definitions of Hochschild cohomology and Lie algebra cohomology.


\section*{Appendix C: Zero Obstructions}
\addcontentsline{toc}{section}{Appendix C: Zero Obstructions}

Assuing we have an extension of associative algebras $(B, \beta)$ of $K$ by $A$. All essentials are pictured in the following diagram:
\[
\begin{tikzcd}
  &0 \ar{r} &Inn(K) \ar[tail]{r}{j}  &Mul(K) \ar[two heads]{r}{\natural} &Out(K) \ar{r} &0 \\
	&0 \ar{r} &K \ar[tail]{r}{\alpha} \ar{u}{\epsilon} &B \ar[two heads]{r}{\beta} \ar{u} &A \ar{r} \ar[bend left, yshift=0.6ex]{l}{\gamma} \ar{u}{\xi} \ar[swap, dashed]{ul}{\mu^\gamma} &0 \\
	&&&&& A\otimes A \ar[swap]{uul}{R^\xi} \ar[bend left, dashed, yshift=0.3ex]{ulll}{R^\gamma}
\end{tikzcd}
\]

Maps $\alpha, \beta, \epsilon, j$ and $\natural$ are morphisms of algebras.

Define $\mu^\gamma_a: A \rightarrow Mul(K)$ such that 
\[
\mu^\gamma_a(k)=\alpha^{-1} \big(\gamma(a)\cdot \alpha(k), \alpha(k)\cdot \gamma(a)\big)
\]

Or equivalently,

\[
\alpha (\mu^\gamma_a)=\big(\gamma(a)^1, \gamma(a)^2 \big)
\]

\begin{lem}
For the curvature $R^{\mu^\gamma}: A\otimes A \rightarrow Inn(K)$, there exists a unique? lift $R^\gamma$ such that 
\[
R^{\mu^\gamma}=\epsilon \circ R^\gamma
\]
\end{lem}

\begin{proof}
Since $R^{\mu^\gamma}$ takes values in $Inn(K)$ and $\epsilon$ is an epimorphism, the existence of $R^\gamma: A\otimes A \rightarrow K$ follows as before. Under the morphism of $\alpha$, we have
\begin{align*}
\alpha \bigl(R^{\mu^\gamma}(a_1 \otimes a_2)\bigr)&=\alpha (\mu^\gamma(a_1) \mu^\gamma(a_2)-\mu^\gamma(a_1 a_2))   \\
&=\bigl(\gamma^1(a_1)\gamma^1(a_2)-\gamma^1(a_1 a_2),\gamma^2(a_1)\gamma^2(a_2)-\gamma^2(a_1 a_2)\bigr)
\end{align*}
Define $R^\gamma(a_1\otimes a_2)=\gamma(a_1)\gamma(a_2)-\gamma(a_1 a_2)$. Then $\epsilon \circ R^\gamma=(R^{\gamma^1}, R^{\gamma^2})=(\gamma^1 \gamma^1-\gamma^1, \gamma^2 \gamma^2-\gamma^2)=R^{\mu^\gamma}$.
\end{proof}

Denote $h^\gamma=R^\gamma$. This is the hindrance determined by $\gamma$ which does not clearly display in [Hoch46].

Another way to introduce the hindrance above is to use the notion of \emph{produced connection} in $Mul(K)$. See more in [Mack05]. For $a\in A$, if $\gamma$ is a connection in $A$, then there exists a linear mapping $\mu^\gamma$ such that $\epsilon \circ \gamma =\mu^\gamma$. So we have 
\begin{align*}
&\epsilon \circ h^\gamma(a_1\otimes a_2)=\mu^\gamma (a_1) \mu^\gamma(a_2)-\mu^\gamma(a_1 a_2)=R^{\mu^\gamma} (a_1\otimes a_2)   \\
&\epsilon \circ h^\gamma(a_1\otimes a_2)=\epsilon \circ \gamma (a_1) \epsilon \circ \gamma(a_2)-\epsilon \circ \gamma (a_1 a_2)
\end{align*}

Hence we have $h^\gamma (a_1\circ a_2)=\gamma (a_1)\gamma (a_2)-\gamma(a_1 a_2)$ to be the hindrance to our coupling in a same way. 

\begin{prop}
Coupling $\xi$ does not depend on any particular choice of the linear mapping $\gamma$. Every extension uniquely determines a bimultiplication law that covers $\xi$.
\end{prop}
\begin{proof}
Let $\gamma'$ be another linear mapping of $A$ into $B$ such that $\beta\gamma'=id_A$. We would like to show that $\natural \circ \mu^{\gamma'}=\natural \circ \mu^\gamma$. To do this, write $\gamma'=\gamma+\alpha \circ l$ for some maps $l:A\rightarrow K$. Then passing through the surjection $\epsilon$ and $\gamma$ we have
\begin{align*}
\natural(\mu^{\gamma'})
&=\natural \circ (\epsilon \circ \gamma') \\
&=\natural \circ \epsilon(\gamma+\alpha \circ l)  \\
&=\natural(\mu^\gamma)+ \natural(\epsilon \circ \alpha \circ l)  \\
&=\natural(\mu^\gamma) \\
&=\xi
\end{align*}
since $\alpha$ is injective and therefore $\natural \circ \epsilon$ carries $\alpha(K)$ into $Inn(K)$ and leads to zero.
\end{proof}

When a split extension of algebra is given, there are possibly many choices of $\gamma$ and they defines, in a one-to-one fashion, different coverings for $\xi$. By proceeding lemma, $\gamma$ thus defines a hindrance $h^\gamma$. As $Obs(\xi)$=$f(\mu^\gamma, h^\gamma)$, we conclude that $\gamma$ actually determines the obstruction class.

Coupling induced by some extensions is called \textbf{special}, due to Hochschild. 

\begin{lem}
(Necessity)\quad For every $\gamma$ derived from an extension of algebras, $f(\mu^\gamma, h^\gamma)=0$.
\end{lem}
\begin{proof}
\begin{align*}
f(\mu^\gamma, h^\gamma)&=\Delta^\alpha(\mu^\gamma) h^\gamma(a_1\otimes a_2\otimes a_3)\in B  \\
&=u_{a_1}^\gamma h^\gamma(a_2\otimes a_3)-h^\gamma(a_1 a_2\otimes a_3)+h^\gamma(a_1\otimes a_2 a_3)-v_{a_3}^\gamma h^\gamma(a_1\otimes a_2)    \\
&=\alpha(u_{a_1}^\gamma)\gamma(a_2)\gamma(a_3)-\alpha(u_{a_1}^\gamma)\gamma(a_2 a_3)-\gamma(a_1 a_2)\gamma(a_3)     \\
&+\gamma(a_1 a_2 a_3)+\gamma(a_1)\gamma(a_2 a_3)-\gamma(a_1 a_2 a_3)-\alpha(v_{a_3}^\gamma)\gamma(a_1)\gamma(a_2)+\alpha(v_{a_3}^\gamma)\gamma(a_1 a_2)  
\end{align*}

Since $\alpha(u_{a_1}^\gamma)(\cdot)=\gamma(a_1)\cdot_\mu (\cdot)$ and $\alpha(v_{a_3}^\gamma)=(\cdot) \cdot_\mu \gamma(a_1)$, then
\begin{align*}
f(\mu^\gamma,h^\gamma)&=\gamma(a_1)\gamma(a_2)\gamma(a_3)-\gamma(a_1)\gamma(a_2 a_3)-\gamma(a_1 a_2)\gamma(a_3)   \\
&+\gamma(a_1)\gamma(a_2 a_3)-\gamma(a_1)\gamma(a_2)\gamma(a_3)+\gamma(a_1 a_2)\gamma(a_3)  \\
&=0 
\end{align*}
\end{proof}

\begin{thm}
An $A$-kernel $(K, \xi)$ is extendible if and only if $Obs(\xi)=0$. 
\end{thm}

\begin{lem}
(Sufficiency)\quad Given $A$, $K$ and $\xi$. Let $\mu$ be a law that covers $\xi$ and there is a bilinear map $R: A\otimes A\rightarrow K$ being the lift of $\mu$ such that $f(\mu, R)=0$, then \\
1) The algebras $K$ and $A$ form an extension $A'$ such that $A'=K \rtimes_{\mu} A$, \\
2) For this extension we can find a linear mapping $\gamma$ making it split such that $\mu^\gamma=\mu$ and $h^\gamma=R$.
\end{lem}
\begin{proof}
Let the underlying vector space of $A'$ be the direct sum of the underlying vector space of $A$ and $K$. The multiplication on $A'$ is defined by 
\[
(a_1,k_1 )(a_2,k_2 )=(a_1 a_2, \ a_1\cdot k_2+k_1 \cdot a_2+k_1 k_2+h(a_1 \otimes a_2))
\]
Now let us compute
\[
\big((a_1,k_1)(a_2,k_2)\big)(a_3,k_3)=\big(a_1 a_2, \  a_1k_2+k_1\cdot a_2+k_1 k_2+h(a_1\otimes a_2)\big)(a_3,k_3)
\]
The first coordinate is $a_1 a_2 a_3$, and the second one is
\begin{multline*}
(a_1 a_2)\cdot k_3 +\big(a_1\cdot k_2+k_1\cdot a_2+k_1 k_2+h(a_1\otimes a_2)\big)\cdot a_3    \\
+\big(a_1\cdot k_2+k_1\cdot a_2+(k_1 k_2)\cdot a_3+h(a_1\otimes a_2)\big) k_3+h(a_1\otimes a_2 \otimes a_3)  
\end{multline*}
\begin{multline*}
=(a_1 a_2)\cdot k_3+(a_1\cdot k_2)\cdot a_3+(k_1\cdot a_2)\cdot a_3+(k_1 k_2) \cdot a_3+h(a_1\otimes a_2)\cdot a_3   \\
+(a_1\cdot k_2)\cdot k_3+(k_1\cdot a_2)\cdot k_3+k_1 k_2 k_3+h(a_1\otimes a_2) k_3+h(a_1\otimes a_2 \otimes a_3) 
\end{multline*}
On the other hand, we compute 
\[
(a_1,k_1)\big((a_2,k_2 )(a_3,k_3)\big)=(a_1,k_1)\big(a_2 a_3,a_2\cdot k_3+k_2\cdot a_3+k_2 k_3+h(a_2 \otimes a_3)\big)  
\]
Again, the first coordinate is $a_1 a_2 a_3$ and the second one is
\begin{multline*}
a_1 \cdot \big(a_2\cdot k_3+k_2\cdot a_3+k_2 k_3+h(a_2 \otimes a_3)\big)+k_1\cdot (a_2 a_3)  \\
+k_1\cdot \big(a_2\cdot k_3+k_2\cdot a_3+k_2 k_3+h(a_2 \otimes a_3)\big)+h(a_1 \otimes a_2 \otimes a_3) 
\end{multline*}
\begin{multline*}
=a_1\cdot (a_2\cdot k_3)+a_1\cdot (k_2\cdot a_3)+a_1\cdot (k_2 k_3 )+a_1\cdot h(a_2 \otimes a_3)  \\
+k_1\cdot (a_2 a_3)+k_1\cdot (a_2\cdot k_3)+k_1\cdot (k_2\cdot a_3)+k_1 k_2 k_3+k_1 h(a_2 \otimes a_3)+h(a_1 \otimes a_2 \otimes a_3)  
\end{multline*}
By the identities $(3.1)$ and $(3.2)$, we have
\begin{align*}
a_1\cdot (a_2\cdot k_3)&=(a_1 a_2 )\cdot k_3+h(a_1,a_2) k_3  \\
(k_1\cdot a_2 )\cdot a_3&=k_1\cdot (a_2 a_3 )+k_1 h(a_2\otimes a_3)
\end{align*}
By the regularity of $\mu$, we have
\[
(a_1\cdot k_2)\cdot a_3=a_1 \cdot (k_2\cdot a_3)
\]
By the definition of $Mul(K)$, we have
\begin{align*}
v(k_1 k_2)=(k_1 k_2)\cdot a_3&=k_1 (k_2\cdot a_3)=k_1 v(k_2) \\
u(k_2 k_3)=a_1\cdot (k_2 k_3)&=(a_1\cdot k_2 ) k_3=u(k_1) k_2 \\ 
k_1 u(k_3)=k_1\cdot (a_2\cdot k_3)&=(k_1\cdot a_2)\cdot k_3=v(k_1) k_3
\end{align*}
All of them, together with the term $k_1 k_2 k_3$, are identical. The remaining part is
\[
h(a_1 \otimes a_2)\cdot a_3+h(a_1\otimes a_2 \otimes a_3)=a_1\cdot h(a_2\otimes a_3)+h(a_1\otimes a_2\otimes a_3)
\]
This implies $(\delta h)(a_1\otimes a_2\otimes a_3)=f(a_1\otimes a_2\otimes a_3)=0$ as desired. Thus, $A'$ becomes an associative algebra.
Lastly, we identify the subalgebra $(0,K)$ with $K$, and define the homomorphism $\beta \big((a,k)\big)=a$. Let its inverse be 
\[
\gamma(a)=(a,0)
\]
Indeed, $\beta \gamma=id_A$ and $\gamma(a)$ produces the bimultiplication $\mu_a$. Specifically,
\begin{align*}
\alpha \big(\mu_a^{\gamma} (k) \big)&=\big(\gamma(a)\cdot \alpha(k), \alpha(k)\cdot \gamma(a) \big) \\
&=\big((a,0),(0,k),(0,k)(a,0) \big) \\
&=\Big(\big(0,a\cdot k+h(a\otimes 0)\big),\big(0,k\cdot a+h(0\otimes a)\big)\Big) \\
&=(a\cdot k, k\cdot a) \\
&=\mu_a  
\end{align*}
which is identical to $\mu$.
Therefore, we have shown that a coupling having trivial obstruction class determines an extension of algebra.
\end{proof}


\section*{Appendix D: Proof of Theorem 3}
\addcontentsline{toc}{section}{Appendix D: Proof of  Theorem 3}

\begin{proof}
We start with the triple $(A, M, f)$ and build a pair $(K, \xi)$ together with some proper $\mu$ and $h$ such that the following equalities hold:
\begin{align*}
&AnniK=M, \quad \xi=\natural \circ \mu, \quad \epsilon \circ h=R^\mu  \\
F^\xi&=F(\mu, h)=\Delta^\mu h  \\
     &\equiv f
\end{align*}
We define
\[
K=M \oplus L
\]
where
\[  
L:=J \oplus A\otimes A \otimes A^*
\]
such that $J \triangleleft L$ and $A^*=A\oplus 1$ in the underlying vector space 
\[ 
J:=C \oplus I 
\]
\begin{align*}
&C:=\mathbb{F}e\oplus \mathbb{F}f=E \oplus F \\
&I:=E\otimes A' \oplus E\otimes A' \otimes A' \oplus E \otimes A' \otimes A' \otimes A'
\end{align*}

On $C$ the multiplication is defined by 
\[ 
e^2=e,\quad f^2=f,\quad ef=f,\quad fe=e 
\] 
On $I$ and $IC$ the multiplication is defined by
\[
II=0, \quad IC=0
\]
On $CI$ the multiplication is defined by
\[
ev=v=fv \quad \text{for all} v\in I
\]
Because of $J=I \oplus C$ we have in total
\[
IJ=0
\]
On $(C\oplus I)(A\otimes A\otimes A^*)=J(A\otimes A\otimes A^*)$, we trivialize some of the components:
\[
(F \oplus E\otimes A'\otimes A' \oplus E\otimes A'\otimes A'\otimes A')(A\otimes A \otimes A^*)=0,
\]
and concretize the rest ones, $(E \oplus E\otimes A')(A\otimes A \otimes A^*)$, by claiming multiplications between basis: 
\begin{align*}  
e(a_1\otimes a_2 \otimes1)&=e\otimes a_1\otimes a_2  \\
& \in E\otimes A'\otimes A',  \\
e(a_1\otimes a_2\otimes a_3)&=e\otimes a_1\otimes a_2a_3-e\otimes a_1a_2\otimes a_3+e\otimes a_1\otimes a_2\otimes a_3   \\
&   \in E\otimes A'\otimes A'\oplus E\otimes A'\otimes A'\otimes A',   \\
(e\otimes a_1)(a_2\otimes a_3\otimes1)&=e\otimes a_1\otimes a_2\otimes a_3  \\
&   \in E\otimes A'\otimes A'\otimes A',  \\
(e\otimes a_1)(a_2\otimes a_3\otimes a_4)&=e\otimes a_1\otimes a_2\otimes a_3a_4-e\otimes a_1a_2\otimes a_3\otimes a_4+e\otimes a_1a_2\otimes a_3\otimes a_4     \\
& \in E\otimes A'\otimes A'\otimes A'.
\end{align*}

We require $MK=KM=0$ and thus $M$ is our biannihilator of $K$

We close the last unspecified product by assgining   
\[
(A\otimes A \otimes A^*)L=0 
\]

It can rest assured that these would exhaust all possible multiplications between the components of $K$. Next, we shall define the $\mu$-endomorphisms for $K$ compatible with the conditions for $Mul(K)$. Notice that the most influencing $A$-actions on $K$ stand on $E=A\otimes A\otimes A^* \oplus M$.

For all elements $(p, m)\in E$ we define the two-sided actions by
\begin{align*}
a\cdot \big(a_1\otimes a_2\otimes a_3^{\ast}, m \big) &= \big(aa_1\otimes a_2\otimes a_3^{\ast}-a\otimes a_1a_2\otimes a_3^{\ast} +a\otimes a_1\otimes a_2 a_3^{\ast}, \\ 
& \qquad f(a\otimes a_1\otimes a_2) \cdot a_3^{\ast}+a\cdot m \big),  \\
\big(a_1\otimes a_2\otimes a_3^{\ast}, m \big)\cdot a &=\big( a_1\otimes a_2\otimes a_3^{\ast} a, m\cdot a \big)
\end{align*}

Note that by definition $f(a\otimes a_1\otimes a_2) \cdot 1\in N$. On the subspace $J$ we set the left action on it by
\[
a\cdot J=0 \quad \text{or} \quad A\cdot J=0,
\]
and all right actions on it by
\begin{align*}
   e\cdot a &=e\otimes a\in E\otimes A', \\
	 f\cdot a &=0, \\
   (e\otimes a_1)\cdot a &= e\otimes a_1a+e\otimes a_1\otimes a\in (E\otimes A')\oplus(E\otimes A' \otimes A'),     \\
   (e\otimes a_1\otimes a_2)\cdot a &=e\otimes a_1\otimes a_2a-e\otimes a_1a_2\otimes a+e\otimes a_1\otimes a_2\otimes a \\
&\in(E\otimes A' \otimes A')\oplus (E\otimes A' \otimes A' \otimes A'),     \\
   (e\otimes a_1\otimes a_2\otimes a_3)\cdot a &= e\otimes a_1\otimes a_2\otimes a_3a-e\otimes a_1\otimes a_2a_3\otimes a+e\otimes a_1a_2\otimes a_3\otimes a  \\
&\in E\otimes A' \otimes A' \otimes A'
\end{align*}

Indeed, the following conditions should hold under above $K$-multiplications and $A$-actions:

\begin{align*} 
k_1(a\cdot k_2) &=(k_1\cdot a)k_2   \\
(a\cdot k_1)k_2 &=a\cdot(k_1k_2)    \\
k_1(k_2\cdot a) &=(k_1k_2)\cdot a   \\
\end{align*}
and permutability 
\[
a_1\cdot (k\cdot a_2)=(a_1\cdot k)\cdot a_2,
\]

It is easily to see that on the bimodule $M$ all identities hold immediately.

Take 
\[
\bar{h}(a_1 \otimes a_2):=a_1 \otimes a_2 \otimes 1
\]. Due to \textbf{Lemma 6.1} and \textbf{6.2}, the bimodule $M$ and $A\otimes A \otimes A*$ constitute an extension, from which such a $\bar{h}$  fulfills $\epsilon \circ \bar{h} =R^\mu$ and $\delta_{M\oplus A\otimes A \otimes A*}\bar{h}$ coincides with $f$ inevitably.

Recall that
\begin{align*}
\epsilon \circ h(a_1 \otimes a_2) (k)&=R^\mu(a_1 \otimes a_2)(k) \\
(u_{h(a_1 \otimes a_2)}, v_{h(a_1 \otimes a_2)})(k)&=\mu(a_1)\mu(a_2)-\mu(a_1 a_2)(k),
\end{align*}
Equivalently,  
\begin{align*}
a_1\cdot(a_2\cdot k)-a_1 a_2\cdot k &=(a_1\otimes a_2\otimes 1)k \\  
(k\cdot a_1)a_2-k\cdot (a_1a_2) &=k(a_1\otimes a_2\otimes 1)
\end{align*}

We dirty our hand to plug in the listed terms and get
\[
(a_1\otimes a_2\otimes1)k=0,
\]
whence $a_1\cdot(a_2\cdot k)=a_1 a_2\cdot k$, showing that $K$ is (only) a left $A$-module(however \emph{not} a right one, as its structure is ``hindered" by $v_{h(a_1\otimes a_2)}$!).
\end{proof}

\section*{Appendix E: Proof of Theorem 4}
\addcontentsline{toc}{section}{Appendix E: Proof of Theorem 4}

As we have taken $U=U(\mathfrak{g})$ which plays the role of $A$ in previous theorem, we take again its underlying space $U'$ where $u' \mapsto u$ is a naive isomorphism. The tensor product of $U'$ are the product of vector space and we are able to assign some new multiplications. On the other hand, we shall find another expression of $h$ which then determines a modified structure of $K$ (especially $I\subset K$) and the table of $K$-multiplications and $A$-actions. 

\begin{proof}
In order to differ from $\otimes$, we will replace $\oplus$ by $+$ for visual convenience.
$$K:=M + L,$$ 
$$\text{where} \quad L=J + U\otimes U, $$ 
$$\text{where} \quad J=I+ C,$$
$$\text{where} \quad I=U' + U'\otimes U' $$
$$\text{where} \quad C=\mathbb{F}e+\mathbb{F}f $$

The list of $K$-multiplications:
\begin{align*}
&(U' \otimes U')L= (U \otimes U)L=0 \\
&e(u_1 \otimes u_2)=u_1' \otimes u_2', \quad \mathbb{F}e(U\otimes U)\subset U'\otimes U' \\
&f(u_1 \otimes u_2)=0, \quad \mathbb{F}f(U\otimes U)=0  \\
&u'(u_1 \otimes u_2)=(uu_1)'\otimes (u_2)'-u'\otimes (u_1 u_2)', \quad U'(U\otimes U)\subset U'\otimes U'
\end{align*}

The list of $A$-actions on $K$:
\begin{align*}
U\cdot J&=0 \\
U\cdot (U\otimes U+ M)&=\big(u u_1\otimes u_2-u\otimes u_1 u_2, f(u, u_1, u_2) \big) \\
(U'\otimes U' + U\otimes U + M)\cdot U&=0 \\
e\cdot u&=u', \qquad \mathbb{F}e\cdot U=U' \\
f\cdot u&=0, \qquad \mathbb{F}f\cdot U=0 \\
u_1'\cdot u&=(u_1 u)'+u_1'\otimes u', \qquad U'\cdot U \subset U'\oplus U'\otimes U'
\end{align*}
Note that $U\cdot M$ behaves what it does, and we have altogether $M\cdot U=0$ as prescribed before. 

Now we are going to check all four conditions. The given lists appear to be in huge computation, but it will not dirty our hand too much to justify them. Indeed, we concentrate on the most ``typical pieces" and conclude their trivialness from the left hand sides or the right hand sides, respectively. 

We show that $w\cdot(k_1 k_2)=(w\cdot k_1)\cdot k_2$ for $w\in U(\mathfrak{g})$.

1)If $k_1 \in J=C+I$, then from the list, we have
\begin{align*}
k_1k_2\in JK&=J(M+L)\\
&=JL \\
&=(e, f, U', \mathbf{U'\otimes U'})\mathbf{L}\\
&=(e,f, \mathbf{U'})(\mathbf{J}+U'\otimes U')(J+U\otimes U)  \\
&=(e,\mathbf{f})(J+\mathbf{U\otimes U}) \\
&=eJ+e(U\otimes U) \\
&=J+U'\otimes U' \\
&\subset J
\end{align*}
In other words, as long as $J \triangleleft L$ we have $JL\subset L$. Then in the left hand side it is
$$w\cdot(k_1 k_2)\subset U\cdot J=0,$$
and in the right hand sides:
$$(w\cdot k_1)\cdot k_2\subset (U\cdot J)\cdot k_2=0$$

2) If $k_1 \in U\otimes U+M$, the from the left we have 
$$k_1 k_2\subset (U\otimes U+M)K=(U\otimes U)(M+L)=0$$
due to the list, and from the right
\begin{align*}
(w\cdot k_1)&\subset U\cdot(U\otimes U+M) \\
&\subset U'\otimes U'
\end{align*}
Then
\begin{align*}
(U'\otimes U')\cdot k_2 & \\
&\subset (U'\otimes U')K \\
&=(U'\otimes U')L=0
\end{align*}

In this way, we have proved the equality $w\cdot(k_1 k_2)=(w\cdot k_1)\cdot k_2$. Similarly for the other two equalities.

3) Now we need to show that $(k\cdot w_1)w_2-k(w_1 w_2)=m(w_1 \otimes w_2)$ (this is for $\epsilon \circ h=R^\mu$!)

Let $k=\mathbf{u'_1}+u'_2\otimes u'_3 + \mathbf{e} + f + u_1 \otimes u_2 +m$. Here only the bold elements $(u'_1, e)\in U'+\mathbb{F}e$ are non-zero due to the list. That is, for any $w_1, w_2 \in U$, in the left we have 
\begin{align*}
&(k\cdot w_1)w_2-k(w_1 w_2)= \\
&=[(u'_1+e)\cdot w_1]\cdot w_2-(u'_1+e)(w_1 w_2) \\
&=[(u_1 w_1)'+u'_1\otimes w'_1+w'_1]\cdot w_2-[(u_1 w_1 w_2)'+u_1'\otimes (w_1 w_2)'+(w_1w_2)'] \\
&=[\mathbf{U'} + U'\otimes U' + \mathbf{U'}]\cdot U-\cdots  \\
&=(u_1w_1w_2)'+(u_1w_1)' \otimes w'_2+(w_1 w_2)'+w'_1 \otimes w'_2-(u_1 w_1 w_2)'-u'_1\otimes (w_1 w_2)'-(w_1 w_2)'  \\
&=(u_1w_1)' \otimes w'_2+w'_1 \otimes w'_2-u'_1\otimes (w_1 w_2)',
\end{align*}

while in the right:
\begin{align*}
&m(w_1 \otimes w_2)= \\
&=(u'_1+e)(w_1 \otimes w_2)  \\
&=(uw_1)'\otimes w_2-u'_1\otimes (w_1 \otimes w_2)'+w'_1\otimes w'_2
\end{align*}

Therefore, we get that $(k\cdot w_1)w_2-k(w_1 w_2)=m(w_1 \otimes w_2)$(meaning $K$ is \emph{not} an right $A$-module!). Similar computation for the left one, which claims a left module on $K$, also valid in \textbf{Appendix D})(Tips: always catch the nonzero parts and try to cancel them)

4) We show $AnniK=M$ or $AnniL=0$

Let $l=\alpha e+\beta f+p+i+j \in \mathbb{F}e+\mathbb{F}f+U\otimes U+U'+U'\otimes U'$. By contradiction, we suppose $ZL\neq 0$, then there exists a $l'\in L$ such that $ll'=0$ or $l'l=0$. It suffices to accommodate a $l'$ deviously.

Since the arbitrariness of $l'$, at first we set $l'\in U\otimes U$ and we multiply from the left by
$$(\alpha e+\beta f+p+i+j)l'\in l(U\otimes U)=0$$
Then from the list we have
$$(\alpha e+i)(u_1 \otimes u_2)=0$$
The linear independence of the basis implies
$$\alpha e+i=0 (*)$$ 

Next, let $l'=f$ and multiply from the right
$$f(\alpha e+\beta f+p+i+j)=0$$
We get
$$\alpha e+\beta+fv=\alpha e+\beta+v=0$$
Again,the linear independence of the basis implies
$$\alpha, \beta, v=0$$

Since (*) we have $i=0$, then immediately $v=i+j, j=0$

Finally, set $l'=e$ and multiply from the left  
$$(p+j)e=p\cdot e=0$$
It follows $e=0$. Therefore all constituents of $l'$ is zero and the claim is disproved, meaning the annihilator of $L$ is nihil.

\end{proof}

\bibliography{test_mod}
\bibliographystyle{plainurl} 

\end{document}